\documentclass[11pt,reqno,a4paper]{amsart}
\usepackage[normalem]{ulem}
\usepackage[colorlinks]{hyperref}
\usepackage{amssymb,epsfig,graphics}
\usepackage{amsmath}
\usepackage{enumerate}
\usepackage{tikz}
\usetikzlibrary{arrows}

\newtheorem{theorem}{Theorem}[section]

\newtheorem{definition}[theorem]{Definition}

\newtheorem{lemma}[theorem]{Lemma}
\newtheorem{sub-lemma}[theorem]{Sub-Lemma}

\newtheorem{remark}[theorem]{Remark}

\def\NN{\mathbb{N}}

\def\RR{\mathbb{R}}

\def\ZZ{\mathbb{Z}}

\DeclareMathOperator*{\esssup}{ess\,sup}
\DeclareMathOperator*{\essinf}{ess\,inf}
\DeclareMathOperator{\osc}{osc}

\let\eps=\varepsilon

\def\RR{{\mathbb R}}
\def\v{{\mathbf v}}
\def\p{{\mathbf p}}
\def\q{{\mathbf q}}
\def\h{{\mathbf h}}
\def\1{{{\mathit 1} \!\!\>\!\! I} }

\renewcommand{\limsup}{\mathop{{\overline {\hbox{{\rm lim}}}}}}

\begin{document}

\title
{Map lattices coupled by collisions: hitting time statistics and collisions per lattice unit}
\author{Wael Bahsoun}
\address{Department of Mathematical Sciences, Loughborough University,
Loughborough, Leicestershire, LE11 3TU, UK}
\email{W.Bahsoun@lboro.ac.uk}
\author{Fanni M. S\'elley}
\address{Mathematical Institute of Leiden University, Niels Bohrweg 1
2333 CA Leiden, The Netherlands}
\email{f.m.selley@math.leidenuniv.nl}
\thanks{The research of W. Bahsoun is supported by EPSRC grant EP/V053493/1. The research of W. B. was partially carried out during a sabbatical supported by a Loughborough University Fellowship scheme. The research of F. M. S\'elley was supported by the European Research Council (ERC) under the European Union’s Horizon 2020 research and innovation programme (grant agreement No 787304). The research of F. M. Sélley was partially carried out at Université de Paris and Sorbonne Université, CNRS, Laboratoire de Probabilités, Statistique et Modélisation, F-75013 Paris, France.}
\keywords{Transfer operators, Coupled map lattices, Collisions, Rare events, Chaos per lattice unit}
\subjclass{Primary 37A05, 37E05}
\begin{abstract}
We study map lattices coupled by collision and show how perturbations of transfer operators associated with the spatially periodic approximation of the model can be used to extract information about collisions per lattice unit. More precisely, we study a map on a finite box of $L$ sites with periodic boundary conditions, coupled by collision.
We derive, via a non-trivial first order approximation for the leading eigenvalue of the rare event transfer operator, a formula for the \emph{first collision rate} and a corresponding \emph{first hitting time law}. For the former we show that the formula scales at the order of $L\cdot\eps^2$, where $\eps$ is the coupling strength, and for the latter, by tracking the $L$ dependency in our arguments, we show that the error in the law is of order  $O\left(C(L)\frac{L\eps^2}{\zeta(L)}\cdot\left|\ln  \frac{L\eps^2}{\zeta(L)}\right|\right)$, where $\zeta(L)$ is given in terms of the spectral gap of the rare event transfer operator, and $C(L)$ has an explicit expression. Finally, we derive an \emph{explicit formula} for the first collision rate \emph{per lattice unit}.
\end{abstract}
\date{\today}
\maketitle
\markboth{Wael Bahsoun and Fanni M. S\'elley}{Map lattices coupled by collisions}
\bibliographystyle{plain}
%\tableofcontents
\section{Introduction}\label{sec:into}
Coupled map lattices are discrete-time dynamical systems  modeling the interaction between microscopic units organized in a spatial structure given by a lattice. The state of each unit, called a \emph{site}, evolves according to the interplay between the local dynamics and the effect of the interaction with other sites. Such systems were introduced in the 80's (see \cite{K92} and references therein) and since then they have seen a remarkable amount of research due to their paramount importance in applications and in studying non-equilibrium thermodynamics \cite{BHLLO15, R12}. From a mathematical viewpoint, the main interest of coupled map lattices is that they provide natural examples of dynamical systems with an infinite dimensional state space. Ergodic and statistical properties of coupled maps were studied extensively since the seminal work of \cite{BS88}. Existence and uniqueness of the SRB measure and decay rate of correlations were the questions studied initially, typically in the weak interaction regime. The early results consider smooth expanding or hyperbolic local dynamics with coupling of similar regularity, while piecewise expanding site dynamics with more general coupling schemes (allowing to treat nearest neighbor coupling, for instance) were considered subsequently. See \cite{CF05,KL06} for an extensive list of references. Considering the complete coupling regime, it is natural to expect phase transition-like phenomena. Such results were only proved in case of finite lattices with a small number of sites \cite{KKN92,S19}, however numerics suggest that the existence of multiple positive Lebesgue measure ergodic components can prevail for large system sizes \cite{F14}.  Another line of current research studies coupled maps in more complex spatial structures. After the pioneering work of \cite{KY10}, more realistic setups were considered in the form of heterogeneous networks \cite{PST20}. 

A natural way to understand a system on an infinite lattice is to consider its spatially periodic approximations, that is, coupled maps defined on finite boxes of sites with periodic boundary conditions. By understanding how various important dynamical quantities, such as entropy \cite{CF13}, escape rates in open systems \cite{BF11}, etc., scale with the size of the system it is possible to meaningfully define the amount of ``chaos'' \emph{per lattice unit}\footnote{See \cite{CE99} for an introduction of this concept.} in the system [see \cite{Y13} section 4.2 for a highlight about the importance of this concept in large systems]. A progress in studying the above quantities, entropy and escape rates, has been obtained in \cite{BF11, CF13} for one-dimensional lattices weakly coupled via a convolution operator.

In this work, we study coupled map lattices where the interaction takes place via rare but intense `collisions' and the dynamics on each site is given by a piecewise uniformly expanding map of the interval. Our main goal is to quantify explicitly the amount of chaos per lattice unit in this model. The peculiarity of the model is that this kind of interaction gives rise to a discontinuous coupling function. Many of the technical challenges this poses were tackled in \cite{KL09}, where the existence of a unique Sinai-Ruelle-Bowen measure and exponential decay of correlations within a suitable class of measures was proved via a decoupling technique\footnote{See also \cite{KL05, KL06} where this technique was successfully applied for other types of coupling.}. The  motivation for this model can be understood through the intimate connections with similar models that are used by mathematicians and physicists in the study of heat transfer and the derivation of an associated Fourier's Law \cite{BGNST,GG08}. 

In this paper, we show how spectral techniques of transfer operators associated with the (finite dimensional) periodic approximation of the model \footnote{In the infinite dimensional system the space is not compact. Obtaining `spectral data' for a transfer operator associated with the infinite dimensional system is not obvious at all. On the other hand, this is well understood for a wide range of finite dimensional systems similar to the coupled system arising from the periodic lattice.} can be used to extract information about collisions per lattice unit. In particular, using this technique, a formula for the \emph{first collision rate} with respect to Lebesgue measure; i.e., the fraction of orbits that see a first collision at each time step of the dynamics\footnote{See Definition \ref{def:rate} for a precise statement.}, and a corresponding \emph{hitting time distribution} for the system of $L$ sites are derived. For the former we show that the formula scales at the order of $ L\cdot\eps^2$, where $\eps$ is the coupling strength, and for the latter we show that the error in the law is of order $O\left( C(L) \frac{L\eps^2}{\zeta(L)}\cdot\left|\ln \frac{L\eps^2}{\zeta(L)}\right|\right)$, where $\zeta(L)$ is given in terms of the speed of mixing, and $C(L)$ is a particular function which is obtained by tracking the $L$ dependency in the abstract perturbation result of \cite{KL09'}. Moreover, we derive an \emph{explicit formula} for the first collision rate \emph{per lattice unit}. To the best of our knowledge explicit formulae of this type did not appear before this work in the mathematical literature. We believe that the transfer operator technique employed in this work can be applied to different types of coupled map lattices addressing a related problem, for instance, to derive a precise formula for the escape rate of the open system studied in \cite{BF11}.

In section \ref{sec:set} we introduce the coupled model, state the assumptions on the local dynamics, introduce the corresponding transfer operators and the function space on which the transfer operators act. We finish this section with the statement of our main results, Theorem \ref{thm:global} and Theorem \ref{thm:perunit}. In section \ref{sec:ex} we provide examples to illustrate Theorem \ref{thm:global}; in particular, we describe in certain low dimensional cases the structure of the set of collision states, its limit set as the coupling strength goes to zero, and the formula of the `derivative' for the dominant eigenvalue of a rare event transfer operator. Section \ref{sec:proofs} contains some technical lemmas, their proofs, and the proofs of Theorems \ref{thm:global} and \ref{thm:perunit}. It also contains a remark discussing the infinite dimensional limit of the collision rate and its per lattice unit counterpart.

\textbf{Acknowledgments.} We would like to thank Gerhard Keller for his comments on an earlier version of this work, which were crucial for the content and the presentation of the paper. We further thank Péter Bálint for inspiring discussions, and an anonymous referee for helpful comments.
\section{Setup}\label{sec:set}
Let $I$ denote the closed unit interval $[0,1]$. Let $X=I^{\Lambda}$, where $\Lambda=\ZZ^d\slash (N\ZZ)^d$ for some integers $N,d \geq 1$. As an ease of notation, we are going to write $L=N^d$ from now on. We will refer to $\Lambda$ as the finite lattice of $L$ sites. We denote by $m_{L}$ the $L$-dimensional Lebesgue measure on $X$.

Let $\tau:I\to I$ be the map describing the local dynamics at a single site. Define $T_0:X\to X$ as the product map $(T_0(x))_\p=\tau(x_{\p})$, $\p\in\Lambda$; i.e., $T_0$ is the uncoupled dynamics on the lattice. 

Let $V^+:=\{e_i\}$ be the standard basis of $\RR^d$ and define $V:=V^+ \cup -V^+$. For $\eps>0$, let $\{A_{\eps,-\v}\}_{\v\in V}\subset I$ be a set of disjoint open intervals, each of length $\eps$. Denote by $(a_\v, a_{-\v})$ the point that  $A_{\varepsilon,\v}\times A_{\varepsilon,-\v}$ shrinks to as $\varepsilon \to 0$ and  $S:=\{(a_\v, a_{-\v})\}_{\v\in V^+}$ be the set of all such points. Following \cite{KL09}, we consider the coupling\footnote{In fact the results of this paper apply to a more general form of coupling, regardless how $\Phi_\eps$ is defined on $\{A_{\eps,-\v}\}_{\v\in V}$, see Remark \ref{rem:new}. We present the particular coupling \eqref{coupling} to link and apply our results to existing models in the literature, such as the one introduced in \cite{KL09}.}
\begin{equation}\label{coupling}
(\Phi_\eps(x))_\p=\begin{cases}
       x_{\p+\v} \quad \text{if } x_\p\in A_{\eps,\v}\text{ and } x_{\p+\v}\in A_{\eps,-\v} \text{ for some } \v\in V\\
       x_\p  \quad\quad\text{otherwise}
       \end{cases}
\end{equation}
where $\p+\v$ is understood modulo $N$ coordinatewise. We define the coupled dynamics $T_\eps:X\to X$ as the composition $T_\eps:=\Phi_\eps\circ T_0$.

We assume that $\tau:I\to I$ is piecewise
monotone and $C^{1+\beta}$ on the intervals $I_1,\dots,I_k$ with an extension of similar smoothness to the boundaries and $|\tau'| > 1$. We further assume that $\tau$ has full branches. It is well known that $\tau$ admits a unique absolutely continuous invariant measure $\mu_{\tau}$ whose density $h$ is bounded from below by a positive number on the whole interval $I$ \cite{BG}, and it is furthermore well known to be (Hölder) continuous. 
%\footnote{We can conclude this for example if we assume $|\tau'| > 2$ and $\tau$ is \emph{weakly covering} in the sense that there exists $K \geq 1$ such that $\bigcup_{n=0}^K\tau^n(I_i)=[0,1]$ for $i=1,\dots,k$; \cite{G12}. Also  note that the above assumptions imply that $\tau$ must be onto.}. 
Moreover, to study first order approximations of collision rates, we assume that $\tau$ is continuous at $a_\v$ and $a_{-\v}$ for all $\v\in V$.

\begin{figure}
\centering
    \begin{tikzpicture}[scale=1.2]
    \draw[fill=black!50!white] (0,0) circle (0.8cm);
    \draw[dotted] (0,0) circle (1cm);
    \draw[fill=black!50!white] (2.5,0) circle (0.8cm);
    \draw[dotted] (2.5,0) circle (1cm);
    \draw[dotted] (5,1)  arc(90:270:1);
    \draw[fill=black!50!white] (5,-0.8) -- (5,0.8) arc(90:270:0.8);
    \draw[dotted] (4.75,-1.56) arc(0:180:1);
    \draw[fill=black!50!white] (2.95,-1.56) -- (4.57,-1.56) arc(0:180:0.8);
    \draw[dotted] (2.25,-1.56) arc(0:180:1);
    \draw[fill=black!50!white] (0.45,-1.56) -- (2.05,-1.56) arc(0:180:0.8);
    \draw[dotted] (0.25,1.56) arc(180:360:1);
    \draw[fill=black!50!white] (0.45,1.56) -- (2.05,1.56) arc(360:180:0.8);
    \draw[dotted] (2.75,1.56) arc(180:360:1);
    \draw[fill=black!50!white] (2.95,1.56) -- (4.55,1.56) arc(360:180:0.8);
    \draw[dotted] (-0.25,1.56) arc(360:270:1);
    \draw[fill=black!50!white] (-1.2,0.76) -- (-1.2,1.56) -- (-0.4,1.56) arc(360:270:0.8);
    \draw[dotted] (-0.25,-1.56) arc(0:90:1);
    \draw[fill=black!50!white] (-1.2,-0.76) -- (-1.2,-1.56) -- (-0.4,-1.56) arc(0:90:0.8);
    \filldraw[black] (1.66,0.62) circle (0.2cm);
    \filldraw[black] (1.98,0.86) circle (0.2cm);
    \filldraw[black] (3.7,0.3) circle (0.2cm);
    \filldraw[black] (2.5,-1.3) circle (0.2cm);
    \filldraw[black] (0.2,-0.98) circle (0.2cm);
    \draw[fill=black] (-0.35,1.56) -- (0.05,1.56) arc(360:180:0.2);
    \end{tikzpicture}
    \caption{The Gaspard-Gilbert model \cite{GG08}. Grey circles represent the obstacles, while the black circles are the moving billiard balls.} \label{fig:GG08}
\end{figure}
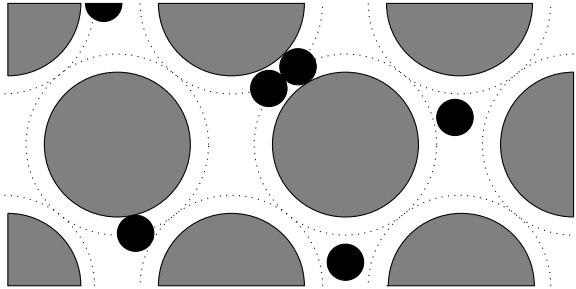

\begin{remark}
Our primary inspiration to study this this type of coupled map lattice was given by the paper \cite{KL09}. In addition to this, an analogy to the billiard model of \cite{GG08} provides further motivation: consider a periodic pattern of smooth scatterers, such as in Figure \ref{fig:GG08}. Among every four of them a cell is formed that contains a moving billiard ball. Thus the cells have the structure of a two-dimensional lattice. The radius of the moving ball is large enough so that the ball is confined to its cell, but small enough so that it is able to collide with the moving balls in each of the four neighbouring cells. When no such interaction happens, the state of each ball (described by its position and velocity) evolves according to classical billiard dynamics. This can be understood as the local dynamics of the sites. When two moving balls collide, they exchange the components of their velocities that are parallel to the line through their centers at the moment of impact. This is similar in spirit to the coupling \eqref{coupling}, where the state of two neighboring sites are interchanged if these states are in an appropriate zone with respect to each other, although in \eqref{coupling} the specific form of the collision is insignificant. As for the local dynamics, following \cite{KL09} we consider the simpler setting of a piecewise uniformly expanding interval maps. The reason for this simplification is that the coupled system will be of similar regularity, but higher dimensional. Transfer operator techniques for multidimensional piecewise expanding maps are well understood \cite{GB89, S00, Li13}, in contrast to high dimensional hyperbolic systems with singularites\footnote{ See the introduction of the recent article \cite{BL21} for an exposition and the current state of the art on transfer operator techniques for (piecewise) hyperbolic systems.}, which would have to be considered in case of the billiard model of \cite{GG08}.  
\end{remark}

In what follows we will build a theory to study the map on the finite lattice of $L$ sites.

\subsection{Transfer operators, first hitting times and first collision rates}

Denote by $L^1_{\RR^{L}}$ the space of measurable functions $f: \RR^{L} \to \RR$ for which
\[
\int_{\RR^{L}} |f| < \infty.
\]
Let $L^\infty$ be the space of essentially bounded measurable functions $f: \RR^{L} \to \RR$.
We will use the notation $\|f\|_1=\int |f|$ for $f \in L^1_{\RR^{L}}$, and $\|f\|_\infty:=\esssup f$. Since $T_0$ is non-singular with respect to $m_L$, we can define the transfer operator associated with $T_0$, denoted by $P_{T_0}$, via the following duality: for $f\in L^1_{\RR^{L}}$ and $g\in L^\infty$
$$\int g\circ T_0\cdot f dm_L=\int g\cdot P_{T_0} f dm_L.$$

Now, define $\mathcal{I}_{L}$ as the set of rectangles $I_{i_1} \times \dots \times I_{i_{L}}$, $i_1,\dots, i_{L} \in \{1,\dots,k\}$ on which $T_0$ is  $C^{1+\beta}$. Then the transfer operator $P_{T_0}$ has the following a.e. pointwise representation:
\[
P_{T_0}f(x)=\sum_{J \in \mathcal{I}_{ L}}\frac{f}{|\det(DT_0)|}\circ (T_0|J)^{-1}(x)1_{T_0(J)}(x).
\]
% Since $\Phi_{\varepsilon}$ is a piecewise permutation of coordinates, we define the transfer operator of $\Phi_{\varepsilon}$ as 
% \[
% P_{\Phi_{\varepsilon}}f(x)=\sum_{y \in \Phi_{\varepsilon}^{-1}(x)}\frac{f(y)}{|\det(D\Phi_{\varepsilon}(y))|)}=\sum_{y \in \Phi_{\varepsilon}^{-1}(x)}f(y).
% \]
% Finally we define the transfer operator of $T_{\varepsilon}$ as 
% \[
% P_{T_{\varepsilon}}=P_{\Phi_{\varepsilon}}P_{T_0}.
% \]

\subsubsection{Rare events and the corresponding transfer operator}
For $\p\in \Lambda$ let 
$$X_{\eps,\v}(\p)=\{x\in X: x_{\p}\notin A_{\eps,\v}\text{ or } x_{\p+\v}\notin A_{\eps,-\v}\},$$
and set 
$$X_{\eps}^0:=\bigcap_{\p\in \Lambda}\bigcap_{\v\in V}X_{\eps,\v}(\p),$$
and
$$H_\eps:=X\setminus X_\eps^0.$$
Let 
$$X^{n-1}_\eps=\cap_{i=0}^{n-1} T_{0}^{-i}X_{\eps}^0.$$
Notice that $X^{n-1}_\eps$ is the set of points whose orbits did not result in any collision up to time $n-1$. 
For $n\ge 1$ and a measurable function $f$, let
\begin{equation}\label{eq:GRET}
\hat{P}^n_{\varepsilon}f= P^n_{T_{0}}(1_{X^{n-1}_{\eps}} f).
\end{equation}
Following the terminology of \cite{K12}, we are going to call $\hat{P}_{\varepsilon}$ the \emph{rare event transfer operator}.

Note that in case of $\eps=0$ we have $X_0^n=X$ for all $n \in \NN$, and hence in this case the operator defined by \eqref{eq:GRET} simply reduces to the transfer operator of $T_{0}$.

It is a natural question to study how long a typical trajectory of the coupled system can survive without experiencing a single collision.
\begin{definition}\label{def:rate}
Let $m$ be a probability measure supported on $X^0_\eps$. Suppose that
$$-\lim_{n \to\infty}\frac{1}{n}\ln m(X_{\eps}^{n-1}):=\hat r$$
exists. Then the quantity in the above limit measures asymptotically, with respect to $m$, the fraction of orbits that see a first collision 
at each time step under the dynamics of $T_\eps$. We call such a quantity the first collision rate under $T_\eps$, with respect to $m$.
\end{definition}
The quantity in Definition \ref{def:rate} captures the fact that the proportion of initial conditions that lead to no collision up to time $n$ shrinks exponentially with rate $\hat{r}$. The motivation behind this definition is to study how $\hat r$, when it exists, scales with the size of the system to meaningfully define the amount of ``chaos'' \emph{per lattice unit}. This is achieved in Theorem \ref{thm:perunit} for the collision models we are interested in.

A natural subsequent question to the fist collision rate is to study how long it takes for a typical trajectory to experience its first collision. We define the \emph{first hitting time} by
$$t_\eps(x)=\inf\{n\ge 0:\,T^n_\eps(x)\in  H_\eps\};$$
i.e., $t_\eps$ is the first time when we see a collision under the coupled map $T_\eps$. Note that $t_\eps$ can be equivalently defined by
$$t_\eps(x)=\inf\{n\ge 0:\,T^n_0(x)\in  H_\eps\}.$$

\begin{remark}\label{rem:new}
The above defined quantities do not take into consideration how the coupling acts on states that involve a collision, since we study trajectories up until the first collision event occurs. More precisely, no coupling takes place on the trajectories we study, so in fact what we study can be viewed as the product map $T_0$ with a specific hole $H_\eps$.\footnote{For more information on systems with holes we refer the reader to \cite{DY} for a basic introduction and a review, to \cite{PU} for a recent account and to \cite{BDT, K12} for their connections with hitting time statistics.}
\end{remark}
The main goal of this work is to understand the first collision rates and hitting time statistics in map lattices coupled by collision as introduced by Keller and Liverani \cite{KL09} and the generalisations of this collision model as described in footnote 1 of \cite{KL09}.

\subsection{Quasi-H\"older spaces}\label{subsec:qh}
We will study the action of the transfer operators defined above on a suitable\footnote{The choice of these spaces among others is almost necessary in this context. Indeed, to apply the perturbation results of \cite{KL09'} in higher dimensions, as in the case of the coupled systems we are dealing with, it is `almost' necessary to use a Banach space that is contained in $L^\infty$.}  Banach spaces. It turns out that for the problem at hand, the most suitable space is the one introduced in \cite{S00}. We start by recalling the function spaces introduced in \cite{S00}.   

Let $0 < \alpha \leq 1$ and $\omega_0 >0$ be parameters. Define the oscillation of $f$ on a Borel set $S \subset \RR^{L}$ as 
\[
\osc(f,S)=\esssup_S f - \essinf_S f.
\]
Let
\[
|f|_{\alpha}=\sup_{0 < \omega \leq \omega_0} \omega^{-\alpha}\int_{\RR^{L}}\osc(f,B_{\omega}(x))dx
\]
and
\[
V_{\alpha}=\left\{ f \in L^1_{\RR^{L}}: |f|_{\alpha} < \infty  \right\}.
\]
Set $\|f\|_{\alpha}=|f|_{\alpha}+\|f\|_{1}$ for $f \in V_{\alpha}$. Then by [\cite{S00}, Proposition 3.3] 
$(V_{\alpha},\|\cdot\|_{\alpha})$ is a Banach space, and the intersection of the unit ball of $V_{\alpha}$ with the set of functions supported on $X$ is compact in $L^1_{\RR^{L}}$. Moreover, by [\cite{S00}, Proposition 3.4] 
$V_{\alpha}$ is continuously injected in $L^\infty$; in particular, $\forall f\in V_\alpha$
\begin{equation}\label{eq:Sausinfty}
    \|f\|_\infty\le\frac{\max(1,\omega_0^\alpha)}{\gamma_{ L}\omega_0^{L}}\|f\|_\alpha,
\end{equation}
where $\gamma_{L}$ denotes the volume of the unit ball of $\RR^{L}$. As this constant will be important to us on multiple occasions, we introduce the notation
\begin{equation}\label{eq:CL}
C(L,\omega_0)=\frac{\max(1,\omega_0^\alpha)}{\gamma_{ L}\omega_0^{L}}.
\end{equation}
\subsection{Notation, collision states and their limit sets}
Now we introduce some notation that are needed in the statements of our main results. All the notation introduced in this subsection is illustrated through simple examples in section \ref{sec:ex}. Recall that $H_{\varepsilon}=X\backslash X_{\varepsilon}^0$ is the set of points whose state results in a collision. 
Let
\begin{equation} \label{eq:Hpv}
H_{\varepsilon}^{\p,\v}=\{ x \in X: x_{\p} \in A_{\varepsilon,\v}, x_{\p+\v} \in A_{\varepsilon,-\v}\}
\end{equation}
for $\p \in \Lambda, \v \in V^+$. This set denotes the set of points whose state results in a collision at site $\p$ in the direction $\v$. Then
\[
H_{\eps}=\cup_{\p \in \Lambda}\cup_{\v \in V^+}H_{\varepsilon}^{\p,\v}.
\]
Note that this is a finite, but, in general, not a disjoint union
as multiple collisions can happen at the same time.

Note that $H_{\varepsilon}^{\p,\v}$ are $L$-dimensional boxes with a base of area $\eps^2$. Their limiting sets, as $\eps\to 0$, will play a crucial role in the statement and proof of Theorem \ref{thm:global}. Denote these limiting sets by $I_{\p,\v}=\lim_{\varepsilon \to 0}H_{\varepsilon}^{\p,\v}$. Explicitly they are given as
\[
I_{\p,\v}=\{ x \in X: x_{\p}=a_\v, x_{\p+\v}= a_{-\v}\}.
\]
\begin{remark}
It is easy to see that $\lim_{\eps\to 0}m_{L}(H_{\eps}^{\p,\v})=m_{ L}(I_{\p,\v})=0$. This implies that $m_{L}(H_{\eps}) \to 0$ as $\eps \to 0$.
\end{remark}
Recall that for each $\v \in V^+$, $(a_\v, a_{-\v})$ denotes the point that  $A_{\varepsilon,\v}\times A_{\varepsilon,-\v}$ shrinks to as $\varepsilon \to 0$ and let $S:=\{(a_\v, a_{-\v})\}_{\v\in V^+}$ be the set of all such points. Let $S^{per} \subset S$ be the set of periodic points in $S$ with respect to the map $T_{2,0}=\tau \times \tau$. 
For each $(a_{\v},a_{-\v}) \in S^{per}$ define $k(\v)$ as the smallest integer $k$ such that $T_{2,0}^{k+1}(a_{\v},a_{-\v}) \in S$. Let $V_k^+ \subset V^+$ denote the set of directions for which $k(\v)=k$ and $K=\{k(\v)\}_{\v \in V^+}$.  Note that the case when $S^{per}\not=\emptyset$ is interesting and important to study since, as the calculation below shows, the collision rate will be higher when the collision zones are centered around periodic points with higher minimal period\footnote{This is inspired by the works of \cite{BY11,KL09'} where this was observed in the study of escape rates in interval maps with holes.}.

\subsection{Statements of the main results}
In the following theorem $L$ is fixed and we understand $o(1) \to 0$ as $\eps \to 0$.
\begin{theorem}\label{thm:global} Fix $L \in \mathbb{N}$. For $\eps\ge0$ sufficiently small, the operator $\hat{P}_\eps: V_\alpha\to V_\alpha$ admits a spectral gap. In particular $T_0$ admits an absolutely continuous invariant measure $\mu_0=\rho_0\cdot m_L$, with $\rho_0(x)=\Pi_{i=1}^{L} h(x_i)$. Moreover, for $\eps>0$,
\begin{enumerate}
\item $\hat{P}_\eps$  has an eigen-function $\rho_\eps>0$, $\int \rho_\eps dx=1$, corresponding to a simple dominant eigenvalue $\lambda_\eps\in(0,1)$, with $$\lambda_{\varepsilon}=1-\mu_0(H_\eps)\cdot\theta(1+o(1)),$$ where 
$$\mu_0(H_\eps)= L\Xi_{\eps}(1+o(1)) \sim \eps^2 L d$$
such that
\[
\Xi_{\eps}=\sum_{\v \in V^+}\int_{A_{\eps,\v}}\int_{A_{\eps,-\v}}h(x_1)h(x_2)dx_1dx_2,
\]
and $\theta$ is a constant. In particular, 
\vskip 0.1cm
\noindent$\bullet$ if $S^{per} \neq \emptyset$ then\footnote{The formula of $\theta$ below is quite general. The reader is invited to go through the examples of section \ref{sec:ex} to see how it reduces significantly in specific settings.}
\begin{equation}\label{eq:theta}
\theta=1-\frac{1}{\sum_{\v \in V^+}h(a_{\v})h(a_{-\v})} \sum_{k \in K}\sum_{\v \in V_k^+}\frac{h(a_{\v})h(a_{-\v})}{|(\tau^{k+1})'(a_{\v})(\tau^{k+1})'(a_{-\v})|};
\end{equation}
\vskip 0.1cm
\noindent$\bullet$ otherwise $\theta=1$.
\vskip 0.1cm
\item The first collision rate of $T_\eps$ with respect to $m_L$ $$\hat r_L:=-\lim_{n\to\infty}\frac{1}{n}\ln m_L(X_\eps^{n-1})$$ 
exists and
$$\hat r_L=\mu_0(H_\eps) \cdot\theta(1+o(1)).$$
\vskip 0.1cm
\item 
In addition, there exists $C>0$, $\xi_\eps>0$, with $\lim_{\eps\to 0}\xi_\eps=\theta$, such that for all $t>0$
$$\left|\mu_0\Big\{t_\eps\ge\frac{t}{\xi_\eps\mu_0(H_\eps)}\Big\}-e^{-t}\right|\le C(t\vee 1)e^{-t} C(L)\frac{Ld \eps^2}{\zeta(L)}\left|\ln\left( \frac{Ld\eps^2}{\zeta(L)}\right)\right|$$
where $\zeta(L)$ is given in terms of the spectral gap of the rare event transfer operator (see \eqref{eq:zeta}) and $C(L)=C(L,\bar c(\tau) L^{-1/2})$ for $C(\cdot,\cdot)$ defined by \eqref{eq:CL}.
\end{enumerate}
\end{theorem}
\begin{remark} \label{rem:theta1d}
The formula \eqref{eq:theta} simplifies considerably when $d=1$. In that case $V^+=\{1\}$ and $S=\{(a_1, a_{-1})\}$. Supposing that $k$ is the smallest integer such that $T_{2,0}^{k+1}(a_1, a_{-1})=(a_1, a_{-1})$, we get
\[
\theta=1-\frac{1}{|(\tau^{k+1})'(a_1)(\tau^{k+1})'( a_{-1})|},
\]
independently of the invariant density $h$ of $\tau$.
\end{remark}

\begin{remark}\label{rem:it3}
In the third item of Theorem \ref{thm:global} we provide a bound on the first hitting time law that depends on the time scale $t$. We follow the spectral approach of \cite{K12} (inspired by \cite{A04}) and compute the precise $L$ dependence of the error term. We note that the particular time dependent bound we obtain improves as $t \to \infty$, so it is meaningful for large values $t$: our goal with this statement is to give a bound on the decay of the tail of the hitting time distribution. We note that for small values of $t$, a uniform bound such as the one given in \cite{HSV99} could be more appropriate.
\end{remark}

The proof of Theorem \ref{thm:global} follows closely the abstract argument of \cite{KL09'}. However, since we are interested in how dynamical quantities scale with the lattice size $L$, we carefully track the $L$ dependency of various quantities and error terms. Furthermore, we provide a precise formula for $\theta$ in this specific system (sometimes referred to in the literature as the \emph{extremal index}), a result which is interesting on its own right. Tracking the $L$ dependency in the argument of \cite{KL09'} is useful in particular in identifying explicitly the constants $\zeta(L)$ and $C(L)$ in the third item of Theorem \ref{thm:global}. Moreover, such understanding is essential to progress in the direction of computing the infinite dimensional limit of the first collision rate per lattice unit. See Remark \ref{rem:infinite}.

As the second item of Theorem \ref{thm:global} states, we have 
\begin{equation}\label{eq:escape0}
\hat r_L= \mu_0(H_\eps) \cdot\theta(1+o(1)).
\end{equation}
On the other hand the proof of Theorem \ref{thm:global} also shows that
\begin{equation}\label{eq:escape}
\mu_0(H_{\eps})={ L}\Xi_{\eps}(1+o(1))\sim \eps^2 { L}d
\end{equation}
and consequently $\hat r_L$, scales linearly in $L$. 
Thus, to obtain a \emph{first collision rate per lattice unit} equation \eqref{eq:escape} suggests that a normalization by $L$ is needed. This is made precise in our second main theorem.

\begin{theorem}\label{thm:perunit}
For any $L\ge 2$, the first collision rate per lattice unit is given by
  $$\frac{1}{L} \hat r_L= \Xi_{\eps} \cdot\theta(1+o(1))$$
    as $\eps\to 0$, where $\Xi_\eps$ and $\theta$ are as in Theorem \ref{thm:global}.
\end{theorem}
\begin{remark}
Note that the little $o$ notation in Theorem \ref{thm:perunit} depends on $L$,  so taking the limit $L \to \infty$ is not an obvious task. Again, see Remark \ref{rem:infinite}.
\end{remark}

\section{Examples}\label{sec:ex}
In this section we provide examples to illustrate Theorem \ref{thm:global}. In particular, we describe the structure of the set $H_{\eps}$, its limit when $\eps \to 0$ and the formula of $\theta$ \eqref{eq:theta} for some low dimensional finite lattices.
\subsection{One dimensional lattice}
Let $d=1$. We naturally identify the sites $\p_{1},\dots,\p_{L}$ with the numbers $1\dots, L$. Now $V^+=\{1\}$ and $S=(a_{1},a_{-1})$, so to simplify notation we will denote $H_{\eps}^{\p_i,\v}$ as $H_{\eps}^{i}$ and $I_{\p_i,\v}$ as $I_{i}$, as it causes no confusion.  

Suppose $(a_1,a_{-1})$ is a $k+1$ periodic point of $T_{2,0}$. We now study lattices of some small number of sites.
\begin{itemize}
\item $L=2$: $H_{\eps}^{1}=A_{\eps,1}\times A_{\eps,-1}$ and $H_{\eps}^{2}=A_{\eps,-1}\times A_{\eps,1}$ are two disjoint squares, while $I_{1}=(a_{1},a_{-1})$ and $I_{2}=(a_{-1},a_{1})$. Both squares represent the single possible interaction between the two sites. See Figure \ref{fig:d1N2} for an illustration.

\begin{figure}
\centering
\begin{tikzpicture}
\draw (0,0) -- (4,0) -- (4,4) -- (0,4) -- (0,0);
\draw[very thick] (0.7,0) -- (1.6,0);
\draw[very thick] (0.7,-0.1) -- (0.7,0.1);
\draw[very thick] (1.6,-0.1) -- (1.6,0.1);

\draw[very thick] (2,0) -- (2.9,0);
\draw[very thick] (2,-0.1) -- (2,0.1);
\draw[very thick] (2.9,-0.1) -- (2.9,0.1);

\draw[very thick] (0,0.7) -- (0,1.6);
\draw[very thick] (-0.1,0.7) -- (0.1,0.7);
\draw[very thick] (-0.1,1.6) -- (0.1,1.6);

\draw[very thick] (0,2) -- (0,2.9);
\draw[very thick] (-0.1,2) -- (0.1,2);
\draw[very thick] (-0.1,2.9) -- (0.1,2.9);

\draw[very thick] (0.1,2.45)  -- (-0.1,2.45) node[left] {$A_{\eps,-1}$};
\draw[very thick] (0.1,1.15)  -- (-0.1,1.15) node[left] {$A_{\eps,1}$};
\draw[very thick] (2.45,0.1) node[above] {$a_{-1}$} -- (2.45,-0.1) node[below] {$A_{\eps,-1}$};
\draw[very thick] (1.15,-0.1) node[below] {$A_{\eps,1}$} -- (1.15,0.1) node[above] {$a_{1}$};

\draw[dotted] (0,2.45) -- (1.05,2.45) -- (1.15,0);
\draw[dotted] (0,1.15) -- (2.45,1.15) -- (2.45,0);

\draw (0.7,2) -- (0.7,2.9) -- (1.6,2.9) -- (1.6,2) -- (0.7,2);
\draw (2,0.7) -- (2.9,0.7) -- (2.9,1.6) -- (2,1.6) -- (2,0.7);

%\draw[dotted] (0.6,0) -- (0.6,2);
%\draw[dotted] (1.5,0) -- (1.5,2);
%\draw[dotted] (0,2) -- (0.6,2);
%\draw[dotted] (0,2.9) -- (0.6,2.9);

%\draw[dotted] (0,0.6) -- (2,0.6);
%\draw[dotted] (0,1.5) -- (2,1.5);
%\draw[dotted] (2,0) -- (2,0.6);
%\draw[dotted] (2.9,0) -- (2.9,0.6);

\filldraw (1.05,2.45) circle (1pt);
\filldraw (2.45,1.05) circle (1pt);
\draw (1.05,3.3) node {$H_{\eps}^1$};
\draw (3.4,1.05) node {$H_{\eps}^2$};
\draw (1.3,2.6) node {$I_1$};
\draw (2.7,1.25) node {$I_2$};
\end{tikzpicture}
\caption{The state space when $ L=2$: the system with two interacting sites.} \label{fig:d1N2}
\end{figure}
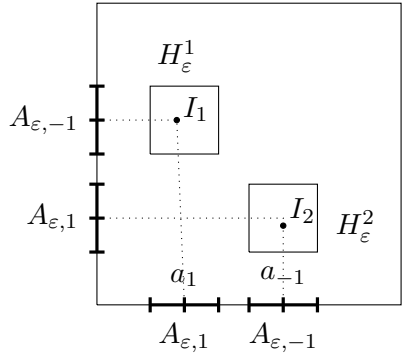

\item $L=3$: $H_{\eps}^{1}=A_{\eps,1}\times A_{\eps,-1}\times I$, $H_{\eps}^{2}=I \times A_{\eps,1}\times A_{\eps,-1}$ and $H_{\eps}^{3}=A_{\eps,-1} \times I \times A_{\eps,1}$, while their limits are the intervals $I_{1}=a_{1}\times a_{-1}\times I$, $I_{2}=I \times a_{1}\times a_{-1}$ and $I_{2}=a_{-1} \times I \times a_{1}$. Notice that since $A_{\eps,1}$ and $A_{\eps,-1}$ are disjoint, the sets $H_{\eps}^i$, $i=1,2,3$ are pairwise disjoint. This translates to the fact that we cannot have multiple interactions simultaneously. The set $H_{\eps}^1$ corresponds to the interaction between sites 1 and 2, $H_{\eps}^2$ to 2 and 3 and $H_{\eps}^3$ to 3 and 1. See Figure \ref{fig:d1N3} for an illustration.
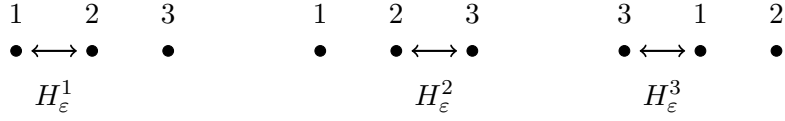
\begin{figure}[ht]
    \centering
    \begin{tikzpicture}
\filldraw (0,0) circle (2pt);
\draw (0,0.5) node {1};
\filldraw (1,0) circle (2pt);
\draw (1,0.5) node {2};
\draw[thick,<->] (0.2,0) -- (0.8,0);
\draw (0.5,-0.6) node {$H_{\eps}^1$};
\filldraw (2,0) circle (2pt);
\draw (2,0.5) node {3};

\filldraw[xshift = 4cm] (0,0) circle (2pt);
\draw[xshift = 4cm] (0,0.5) node {1};
\filldraw[xshift = 4cm] (1,0) circle (2pt);
\draw[xshift = 4cm] (1,0.5) node {2};
\draw[thick,<->,xshift = 4cm] (1.2,0) -- (1.8,0);
\filldraw[xshift = 4cm] (2,0) circle (2pt);
\draw[xshift = 4cm] (2,0.5) node {3};
\draw[xshift = 4cm] (1.5,-0.6) node {$H_{\eps}^2$};

\filldraw[xshift = 8cm] (0,0) circle (2pt);
\draw[xshift = 8cm] (0,0.5) node {3};
\filldraw[xshift = 8cm] (1,0) circle (2pt);
\draw[xshift = 8cm] (1,0.5) node {1};
\draw[xshift = 8cm] (0.5,-0.6) node {$H_{\eps}^3$};
\draw[thick,<->,xshift = 8cm] (0.2,0) -- (0.8,0);
\filldraw[xshift = 8cm] (2,0) circle (2pt);
\draw[xshift = 8cm] (2,0.5) node {2};
\end{tikzpicture}
    \caption{Interactions in case of $L=3$.}
    \label{fig:d1N3}
\end{figure}

\item ${L}=4$: now $H_{\eps}^{1}=A_{\eps,1}\times A_{\eps,-1}\times I \times I$, $H_{\eps}^{2}=I \times A_{\eps,1}\times A_{\eps,-1}\times I$, $H_{\eps}^{3}=I \times I \times A_{\eps,1}\times A_{\eps,-1}$  and $H_{\eps}^{4}=A_{\eps,-1} \times I \times I \times A_{\eps,1}$. Notice that $H_{\eps}^{1}$ and $H_{\eps}^{3}$ have a nonempty intersection, and so does $H_{\eps}^{2}$ and $H_{\eps}^{4}$: 
\begin{align*}
H_{\eps}^{1} \cap H_{\eps}^{3} &=A_{\eps,1}\times A_{\eps,-1}\times A_{\eps,1}\times A_{\eps,-1} \\ 
H_{\eps}^{2} \cap H_{\eps}^{4} &=A_{\eps,-1}\times A_{\eps,1}\times A_{\eps,-1}\times A_{\eps,1}.
\end{align*}
In terms of the coupled system, this means that sites $1$ and $3$ can simultaneously interact with their neighbors to the right, and so can sites $2$ and $4$. See Figure \ref{fig:d1N4}.
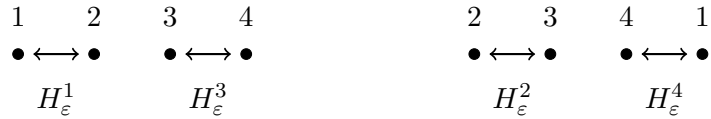
\begin{figure}[ht]
    \centering
    \begin{tikzpicture}
\filldraw (0,0) circle (2pt);
\draw (0,0.5) node {1};
\filldraw (1,0) circle (2pt);
\draw (1,0.5) node {2};
\draw[thick,<->] (0.2,0) -- (0.8,0);
\draw (0.5,-0.6) node {$H_{\eps}^{1}$};
\filldraw (2,0) circle (2pt);
\draw (2,0.5) node {3};
\filldraw (3,0) circle (2pt);
\draw (3,0.5) node {4};
\draw[thick,<->] (2.2,0) -- (2.8,0);
\draw (2.5,-0.6) node {$H_{\eps}^{3}$};

\filldraw[xshift = 6cm] (0,0) circle (2pt);
\draw[xshift = 6cm] (0,0.5) node {2};
\draw[xshift = 6cm] (0.5,-0.6) node {$H_{\eps}^{2}$};
\filldraw[xshift = 6cm] (1,0) circle (2pt);
\draw[xshift = 6cm] (1,0.5) node {3};
\draw[thick,<->,xshift = 6cm] (0.2,0) -- (0.8,0);
\filldraw[xshift = 6cm] (2,0) circle (2pt);
\draw[xshift = 6cm] (2,0.5) node {4};
\filldraw[xshift = 6cm] (3,0) circle (2pt);
\draw[xshift = 6cm] (3,0.5) node {1};
\draw[thick,<->,xshift = 6cm] (2.2,0) -- (2.8,0);
\draw[xshift = 6cm] (2.5,-0.6) node {$H_{\eps}^{4}$};

\end{tikzpicture}
    \caption{Possible simultaneous interactions when $d=1$, ${L}=4$.}
    \label{fig:d1N4}
\end{figure}

\end{itemize}
 According to Remark \ref{rem:theta1d}, we obtain in all three cases studied above
\begin{align*}
\theta=1-\frac{1}{|(\tau^{k+1})'(a_1)(\tau^{k+1})'(a_{-1})|},
\end{align*}
independently of $h$. This is easy to prove when the sets $H_{\eps}^i$ are pairwise disjoint, as in the case of $ L=2$ and $L=3$.

 However, we have to be more careful when the intersections $H_{\eps}^i \cap H_{\eps}^j$, $i \neq j$ are not necessarily empty, although we eventually obtain the same formula. Explaining the reason for this through the example of $L=4$, the key fact the proof will depend on is that the limit sets $I_{1}=a_{1}\times a_{-1}\times I \times I$, $I_{2}=I \times a_{1}\times a_{-1}\times I$, $I_{3}=I \times I \times a_{1}\times a_{-1}$  and $I_{4}=a_{-1} \times I \times I \times a_{1}$ are two dimensional sets, but the intersections
\begin{align*}
I_{1} \cap I_{3} &=a_{1}\times a_{-1}\times a_{1}\times a_{-1} \\ 
I_{2} \cap I_{4} &=a_{-1}\times a_{1}\times a_{-1}\times a_{1}
\end{align*}
are two points with zero two dimensional Lebesgue measure\footnote{In the general case such intersections are not just points, but they will always be negligible relative to the appropriate Lebesgue measure.}, which will prove to be enough for them to be negligible in the computation of $\theta$.

\subsection{Two dimensional lattice}
Let $d=2$ and $L=4$. Denote the sites by $\p_1=(0,0)$, $\p_2=(1,0)$, $\p_3=(0,1)$ and $\p_4=(1,1)$. Then $V^+=\{\v,\h\}$, where $\v=(0,1)$ and $\h=(1,0)$. The set $S$ has two elements: $(a_{\v},a_{-\v})$ and $(a_{\h},a_{-\h})$. Suppose the minimal period of $(a_{\v},a_{-\v})$ is $k+1$, while the minimal period of $(a_{\h},a_{-\h})$ is $\ell+1$. If we study the sets $I_{\p_i,\v_j}$, we find that no three can intersect, but there are eight pairs that can intersect. Namely, $I_{\p_1,\h}$ and $I_{\p_3,\h}$; $I_{\p_2,\h}$ and $I_{\p_4,\h}$; $I_{\p_1,\v}$ and $I_{\p_2,\v}$; $I_{\p_3,\v}$ and $I_{\p_4,\v}$; $I_{\p_2,\h}$ and $I_{\p_3,\h}$;
$I_{\p_1,\h}$ and $I_{\p_4,\h}$;
$I_{\p_1,\v}$ and $I_{\p_4,\v}$ and
$I_{\p_2,\v}$ and $I_{\p_3,\v}$. 
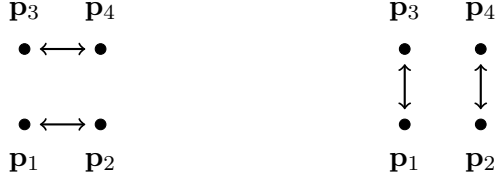
\begin{figure}[h!]
    \centering
    \begin{tikzpicture}
\filldraw (0,0) circle (2pt);
\draw (0,-0.5) node {$\p_1$};
\filldraw (1,0) circle (2pt);
\draw (1,-0.5) node {$\p_2$};
\draw[thick,<->] (0.2,0) -- (0.8,0);
\filldraw (0,1) circle (2pt);
\draw (0,1.5) node {$\p_3$};
\filldraw (1,1) circle (2pt);
\draw (1,1.5) node {$\p_4$};
\draw[thick,<->] (0.2,1) -- (0.8,1);

\filldraw[xshift = 5cm] (0,0) circle (2pt);
\draw[xshift = 5cm] (0,-0.5) node {$\p_1$};
\filldraw[xshift = 5cm] (1,0) circle (2pt);
\draw[xshift = 5cm] (1,-0.5) node {$\p_2$};
\draw[thick,<->,xshift = 5cm] (0,0.2) -- (0,0.8);
\filldraw[xshift = 5cm] (0,1) circle (2pt);
\draw[xshift = 5cm] (0,1.5) node {$\p_3$};
\filldraw[xshift = 5cm] (1,1) circle (2pt);
\draw[xshift = 5cm] (1,1.5) node {$\p_4$};
\draw[thick,<->,xshift = 5cm] (1,0.2) -- (1,0.8);

\end{tikzpicture}
    \label{fig:d2N4}
    \caption{Possible simultaneous interactions when $d=2$, ${L}=4$. Left: corresponds to the intersection of $H_{\eps}^{\p_1,\h}$ and $H_{\eps}^{\p_3,\h}$, $H_{\eps}^{\p_2,\h}$ and $H_{\eps}^{\p_4,\h}$, $H_{\eps}^{\p_1,\h}$ and $H_{\eps}^{\p_4,\h}$ and $H_{\eps}^{\p_2,\h}$ and $H_{\eps}^{\p_3,\h}$; right: corresponds to the intersection of $H_{\eps}^{\p_1,\v}$ and $H_{\eps}^{\p_2,\v}$, $H_{\eps}^{\p_3,\v}$ and $H_{\eps}^{\p_4,\v}$, $H_{\eps}^{\p_1,\v}$ and $H_{\eps}^{\p_4,\v}$ and $H_{\eps}^{\p_2,\v}$ and $H_{\eps}^{\p_3,\v}$.}
\end{figure}

By Theorem \ref{thm:global} we obtain
\begin{align*}
\theta=1&-\frac{1}{h(a_{\v})h(a_{-\v})+h(a_{\h})h(a_{-\h})} \cdot \frac{h(a_{\v})h(a_{-\v})}{|(\tau^{k+1})'(a_{\v})(\tau^{k+1})'(a_{-\v})|} \\
&+\frac{1}{h(a_{\v})h(a_{-\v})+h(a_{\h})h(a_{-\h})} \cdot \frac{h(a_{\h})h(a_{-\h})}{|(\tau^{\ell+1})'(a_{\h})(\tau^{\ell+1})'(a_{-\h})|}.
\end{align*}
So unless there is a close relation between the values at the points $a_{\pm \v}$ and $a_{\pm \h}$ of $h$ and the derivative of $\tau$, this expression depends nontrivially on the density of the site dynamics $\tau$.

\section{Proofs}\label{sec:proofs}
We first compute a formula for the measure of $H_{\eps}$ with respect to the $T_0$-invariant measure $\mu_0=\Pi_{i=1}^{L} h \cdot m$.

\begin{lemma} \label{lem:mu0H}
\[
\mu_0(H_{\eps})= L\Xi_{\eps}(1+o(1))
\]
where
\[
\Xi_{\eps}=\sum_{\v \in V^+}\int_{A_{\eps,\v}}\int_{A_{\eps,-\v}}h(x)h(y)dxdy
\]
and $o(1) \to 0$ as $L\eps^2 \to 0$.
\end{lemma}

\begin{proof}
We start by making some observations that will considerably simplify the formula for $\mu_0(H_{\epsilon})$. First notice that for each $\p \in \Lambda$ we must have  $H_\eps^{\p,\v_i} \cap H_\eps^{\p,\v_j} =\emptyset$ for $\v_i \neq \v_j$, as the state of site $\p$ cannot be simultaneously in $A_{\eps,\v_i}$ and $A_{\eps,\v_j}$. Notice also that at most $ L/2$ sets $H_\eps^{\p_i,\v_i}$ can intersect: indeed, $x \in H_\eps^{\p,\v}$ implies that $x_\p \in A_{\eps,\v}$ for some $\v \in  V^{\pm}$ and $x_{\p+\v}$ is in $A_{\eps,-\v}$ where $-\v \in  V^{\mp}$. This  shows that $x \in H_\eps^{\p,\v}$ implies that $x \notin H_\eps^{\p+\v,\v'}$ for any $\v \in  V^{\pm}$. So we can have $x$ in at most $ L/2$ sets $H_\eps^{\p,\v}$ simultaneously. 

Let $\Lambda^k=\{(\p_1,\dots,\p_k): \p_i \in \Lambda,\: i=1,\dots,k\}$ and $(V^+)^k=\{(\v_1,\dots,\v_k): \v_i \in V^+,\: i=1,\dots,k\}$. Fix an ordering on the elements of $\Lambda$ and denote it by $<$. Then according to de Moivre's inclusion-exclusion principle and using the above observations we can write 
\begin{align}
\mu_0(H_{\eps})&=\sum_{\p \in \Lambda}\sum_{\v \in V^+}\mu_0(H_{\eps}^{\p,\v})-\sum_{\substack{\p_1,\p_2 \in \Lambda \\ \p_1 < \p_2}}\sum_{\v_1,\v_2 \in V^+}\mu_0(H_{\eps}^{\p_1,\v_1} \cap H_{\eps}^{\p_2,\v_2}) \nonumber \\
&+\sum_{\substack{\p_1,\p_2,\p_3 \in \Lambda \\ \p_1 < \p_2  < \p_3}}\sum_{\v_1,\v_2,\v_3 \in V^+}\mu_0(H_{\eps}^{\p_1,\v_1} \cap H_{\eps}^{\p_2,\v_2}\cap H_{\eps}^{\p_3,\v_3})-\dots \nonumber \\
&=\sum_{k=1}^{\lfloor {L}/2 \rfloor}(-1)^{k+1}\sum_{\substack{\{\p_i\}_{i=1}^{k} \in \Lambda^k \\ \p_i < \p_{i+1}}}\sum_{\{\v_i\}_{i=1}^{k} \in (V^+)^k}\mu_0\left(\bigcap_{i=1}^k H_{\eps}^{\p_i,\v_i}\right). \label{eq:pm}
\end{align}
Note that some intersections might be empty, but selecting the non-empty ones would make the formula unnecessarily complicated and would provide little benefit at this point. We also note here that the principle applied to a finite measure of density $F$ with respect to the Lebesgue measure implies that
\begin{align}
\int_{H_{\eps}} F(x) dx&=\sum_{\ell=1}^{\lfloor {L}/2 \rfloor}(-1)^{\ell+1}\sum_{\substack{\{\p_i\}_{i=1}^{\ell} \in \Lambda^{\ell} \\ \p_i < \p_{i+1}}}\sum_{\{\v_i\}_{i=1}^{\ell} \in (V^+)^{\ell}}\int_{\cap_{i=1}^{\ell} H_{\eps}^{\p_i,\v_i}}F(x)dx. \label{eq:pmF}
\end{align}
For $k=1$ we get 
\begin{equation} \label{eq:nxi}
\sum_{\p \in \Lambda }\sum_{\v \in V^+}\mu_0(H_\eps^{\p,\v})={L}\sum_{\v \in V^+}\int_{A_{\eps,\v}}\int_{A_{\eps,-\v}}h(x_1)h(x_2)dx_1dx_2={L}\Xi_{\eps}.
\end{equation} 
where
\[
\Xi_{\eps}=\sum_{\v \in V^+}\int_{A_{\eps,\v}}\int_{A_{\eps,-\v}}h(x_1)h(x_2)dx_1dx_2.
\]
Now let $k > 1$. We remind the reader that we denoted the density of the site dynamics $\tau$ by $h$. Since $h$ is of bounded variation, there exists $\overline{C} > 0$ such that $ h \leq \overline{C}$. If the intersection $\cap_{i=1}^k H_{\eps}^{\p_i,\v_i}$ is nonempty we can write
\begin{align*}
&\mu_0\left(\bigcap_{i=1}^k H_{\eps}^{\p_i,\v_i}\right)=\int_{A_{\eps,\v_1}}\int_{A_{\eps,-\v_1}}\dots\int_{A_{\eps,\v_k}}\int_{A_{\eps,-\v_k}}\prod_{i=1}^{2k}h(x_i)dx_1\dots dx_{2k} \\
&=\mu_0(H_{\eps}^{\p_1,\v_1})\int_{A_{\eps,\v_2}}\int_{A_{\eps,-\v_2}}\dots\int_{A_{\eps,\v_k}}\int_{A_{\eps,-\v_k}}\prod_{i=3}^{2k}h(x_i)dx_3\dots dx_{2k}
\end{align*}
which can be bounded by
\begin{equation} \label{eq:kint}  \mu_0\left(\bigcap_{i=1}^k H_{\eps}^{\p_i,\v_i}\right) \leq \overline{C}^{2(k-1)}  \eps^{2(k-1)}\mu_0(H_{\eps}^{\p_1,\v_1})
\end{equation}
Now first notice that
\[
\mu_0(H_{\eps})=\mu_0\left(\bigcup_{\p \in \Lambda}\bigcup_{\v \in V^+}H_{\eps}^{\p,\v}\right) \leq \sum_{\p \in \Lambda} \sum_{\v \in V^+}\mu_0(H_{\eps}^{\p,\v}) = {L}\Xi_{\eps}
\]
by \eqref{eq:nxi}.

For the lower bound we use \eqref{eq:nxi} for $k=1$ and \eqref{eq:kint} for each $k > 1$:
\begin{align*}
&\sum_{k=1}^{\lfloor {L}/2 \rfloor}(-1)^{k+1}\sum_{\substack{\{\p_i\}_{i=1}^{k} \in \Lambda^k \\ \p_i  < \p_{i+1}}}\sum_{\{\v_i\}_{i=1}^{k} \in (V^+)^k}\mu_0\left(\bigcap_{i=1}^k H_{\eps}^{\p_i,\v_i}\right) \\
&={L}\Xi_{\eps}+\sum_{k=2}^{\lfloor {L}/2 \rfloor}(-1)^{k+1}\sum_{\substack{\{\p_i\}_{i=1}^{k} \in \Lambda^k \\ \p_i < \p_j}}\sum_{\{\v_i\}_{i=1}^{k} \in (V^+)^k}\mu_0\left(\bigcap_{i=1}^k H_{\eps}^{\p_i,\v_i}\right) \\
&\geq {L}\Xi_{\eps}-\left(\sum_{k=2}^{\lfloor {L}/2 \rfloor}\sum_{\substack{\{\p_i\}_{i=1}^{k} \in \Lambda^k \\ \p_i < \p_{i+1}}}\sum_{\{\v_i\}_{i=1}^{k} \in (V^+)^k}\mu_0\left(\bigcap_{i=1}^k H_{\eps}^{\p_i,\v_i}\right)\right) \\
&\geq {L}\Xi_{\eps}-\sum_{k=2}^{\lfloor {L}/2 \rfloor}\sum_{\substack{\{\p_i\}_{i=1}^{k} \in \Lambda^k \\ \p_i  < \p_{i+1}}}\sum_{\{\v_i\}_{i=1}^{k} \in (V^+)^k}\overline{C}^{2(k-1)} \eps^{2(k-1)} \mu_0(H_{\eps}^{\p_1,\v_1}) \\
&= {L}\Xi_{\eps}-\sum_{k=2}^{\lfloor {L}/2 \rfloor}\binom{{L}}{k-1}d^{k-1}\sum_{\p_1 \in \Lambda}\sum_{\v_1 \in V^+}\overline{C}^{2(k-1)} \eps^{2(k-1)}\mu_0(H_{\eps}^{\p_1,\v_1}) \\
&= {L}\Xi_{\eps}-\left (\sum_{\p_1 \in \Lambda}\sum_{\v_1 \in V^+} \mu_0(H_{\eps}^{\p_1,\v_1}) \right )\sum_{k=2}^{\lfloor {L}/2 \rfloor}\binom{{L}}{k-1}d^{k-1}\overline{C}^{2(k-1)}\eps^{2(k-1)},
\end{align*}
which can be written as
\begin{align*}
\mu_0(H_{\eps}) \geq {L}\Xi_{\eps}\left(1-\sum_{k=2}^{\lfloor {L}/2 \rfloor}\binom{{L}}{k-1}d^{k-1}\overline{C}^{2(k-1)} \eps^{2(k-1)} \right).
\end{align*}
The first term in the sum $\sum_{k=2}^{\lfloor {L}/2 \rfloor}\binom{{L}}{k-1}d^{k-1}\overline{C}^{2(k-1)}\eps^{2(k-1)}$ is ${L}d\overline{C}^2\eps^2$. This is the leading term assuming ${L}\eps^2$ is sufficiently small, so
\[
\mu_0(H_{\eps}) \geq {L}\Xi_{\eps}(1-C{L}d\eps^2)={L}\Xi_{\eps}(1+o(1)).
\]
\end{proof}

We now turn to the study of the transfer operators $\hat P_{\eps}$. Observe that $T_0=\tau \times \dots \times \tau$ satisfies conditions (PE1)-(PE4) of \cite{S00} for any $\alpha \leq \beta$. Indeed, the sets $U_i$ are given by the rectangles $I_{\underline{i}}=I_{i_1}\times\dots\times I_{i_{ L}}$ where $\underline{i}=(i_1,\dots,i_{L}) \in \{1,\dots,k\}^L$. We further observe that since $\tau$ has full branches, that is, $\tau(I_{j})=I$ for each $j=1,\dots,k$, we have $T_0(I_{\underline{i}})=I^L$ for all $\underline{i} \in \{1,\dots,k\}^L$. We are going to use the notation $T_{0,\underline{i}}^{-1}=(T_0|_{I_{\underline{i}}})^{-1}$ for the inverse branches, defined on $I^L$. On the sets $I_{\underline{i}}$, the map $T_0$ is $C^{1+\beta}$ and can be extended to the boundary as a function of the same smoothness. This means that the first requirement of (PE2) about the piecewise smoothness of $T_0$ is satisfied, while the second requirement asking that for all $\underline{i}$, $\omega \leq \omega_0$, $z \in I^L$ and $x,y \in B_{\omega}(z) \cap I^L$
 \begin{equation} \label{eq:c}
 |\det D_x T_{0,\underline{i}}^{-1}-\det D_y T_{0,\underline{i}}^{-1}| \leq  c_L|\det D_z T_{0,\underline{i}}^{-1}|\omega^{\alpha},
 \end{equation}
 also holds, and we can choose $c_L=c(\tau)\sqrt{L}$. Finally, (PE4) requiring that for all $x,y \in I^L$
 \begin{equation} \label{eq:s}
 d(T_{0,\underline{i}}^{-1}(x),T_{0,\underline{i}}^{-1}(y)) \leq sd(x,y)
 \end{equation}
 holds with $s:=\frac{1}{\min|\tau'|}$.
\begin{lemma}\label{lem:A2A3}
Let $n \in \mathbb{N}$ and $\varepsilon \geq 0$. There exist constants $\sigma\in(0,1)$, $ D(L,\eps, \omega_0) > 0$ such that
\begin{align}
\|\hat{P}_{\varepsilon}^nf\|_1 &\leq \|f\|_1 \label{eq:LY1} \\
\|\hat{P}_{\varepsilon}^nf\|_{\alpha} &\leq \sigma^n \|f\|_{\alpha}+ D(L,\eps, \omega_0)\|f\|_1. \label{eq:LY2}
\end{align}
\end{lemma}

\begin{proof}
Since
\[
\|\hat{P}_{\varepsilon}f\|_1=\int_X |P_{T_0}(1_{X_{\varepsilon}^0}f)| \leq \int_X |1_{X_{\varepsilon}^0}f| \leq \int_X |f|,
\]
we have $\|\hat{P}_{\varepsilon}^nf\|_1 \leq \|f\|_1$ for all $n \in \mathbb{N}$ and $\varepsilon \geq 0$.

To prove \eqref{eq:LY2} we show the following:
\begin{align}
\|P_{T_0}f\|_{\alpha}&\leq \sigma \|f\|_{\alpha}+K\|f\|_1 \label{eq:LY2a} ;\\
\|1_{X_{\varepsilon}^0}f\|_{\alpha}&\leq \|f\|_{\alpha}+ 8 Ld\varepsilon\omega_0^{1-\alpha}\|f\|_1, \label{eq:LY2b} 
\end{align}
Writing $B(L,\eps,\omega_0)=\sigma\cdot 8 Ld\varepsilon\omega_0^{1-\alpha}+K$, the inequalities \eqref{eq:LY2a} and \eqref{eq:LY2b} imply that
\begin{equation}\label{eq:LYoneit}
\|\hat{P}_{\varepsilon}f\|_{\alpha} \leq \sigma\|f\|_{\alpha}+ B(L,\eps,\omega_0)\|f\|_1,
\end{equation}
which proves \eqref{eq:LY2} by iteration with $D(L,\varepsilon,\omega_0)=B(L,\varepsilon,\omega_0)\sum_{k=0}^{\infty}\sigma^k$. 

 We first prove \eqref{eq:LY2a}, arguing similarly as in Lemma 4.1 of \cite{S00}, but in a simplified setting since $T_0$ is piecewise onto\footnote{In fact we include a detailed proof of \eqref{eq:LY2a} instead of using the general statement of Lemma 4.1 of \cite{S00} two reasons. First to track the $L$ dependency in all the constants. Second to show how assumptions on $\tau$ make it possible to avoid  competition between the $L$ dependent complexity (which appears in the general argument of \cite{S00}, see (PE5) and Lemma 2.1 of \cite{S00}) and the expansion of $\tau$. Indeed, had we used the general statement of Lemma 4.1 of \cite{S00}, we would have the following constraint: $s^{\alpha}+\frac{4s}{1-s}\cdot \frac{L(L-3)}{2\sqrt{\pi}} < 1$, when $ L> 3$, which obviously would put a restriction on $L$ considering $s$ fixed.}.
\begin{align*}
\osc(P_{T_0}f,B_{\omega}(x)) &\leq \sum_{\underline{i}}\osc\left(\left(\frac{f}{|\det DT_0|} \right) \circ T_{0,\underline{i}}^{-1} \cdot \mathbf{1}_{T_0 (I_{\underline{i}})},B_{\omega}(x) \right) \\
&=\sum_{\underline{i}}\osc\left(\left(\frac{f}{|\det DT_0|} \right) \circ T_{0,\underline{i}}^{-1},B_{\omega}(x) \right) \\
&\leq \sum_{\underline{i}}\osc\left(\frac{f}{|\det DT_0|} ,I_{\underline{i}} \cap T_0^{-1}B_{\omega}(x) \right) \\
&\leq \sum_{\underline{i}}\osc\left(\frac{f}{|\det DT_0|} ,I_{\underline{i}} \cap B_{s\omega}(y_{\underline{i}}) \right),
\end{align*}
setting $y_{\underline{i}}=T_{0,\underline{i}}^{-1}(x)$ and remembering that $s$ was defined as the reciprocal of the minimum expansion of $\tau$ (and consequently $T_0$). It is easy to prove that for $g_1, g_2 \in L^{\infty}$, $g_2>0$ and some $S$ Borel set
\[
\osc(g_1g_2,S) \leq \osc(g_1,S)\esssup_S g_2 + \osc(g_2,S)\essinf_S |g_1|
\]
holds. (The reader is invited to check Proposition 3.2 (iii) in \cite{S00}.) This implies that
\begin{align*}
\osc(P_{T_0}f,B_{\omega}(x)) &\leq \sum_{\underline{i}}\osc\left(f,B_{s\omega}(y_{\underline{i}}) \right) \esssup_{I_{\underline{i}}\cap B_{s\omega}(y_{\underline{i}})}\frac{1}{|\det DT_0|}\\
&\quad +\osc\left(\frac{1}{|\det DT_0|},I_{\underline{i}}\cap B_{s\omega}(y_{\underline{i}}) \right) \essinf_{ B_{s\omega}(y_{\underline{i}})}|f| \\
&\leq \sum_{\underline{i}}(1+ c_Ls^{\alpha}\omega^{\alpha})\osc\left(f,B_{s\omega}(y_{\underline{i}}) \right)\frac{1}{|\det DT_0|}(y_{\underline{i}})\\
&\quad + c_Ls^{\alpha}\omega^{\alpha}\frac{1}{|\det DT_0|}(y_{\underline{i}}) |f|(y_{\underline{i}}) \\
&\leq (1+ c_Ls^{\alpha}\omega^{\alpha})P_{T_0}(\osc(f,B_{s\omega}( x)))+ c_Ls^{\alpha}\omega^{\alpha}P_{T_0}|f| (x),
\end{align*}
where $c_L$ and $s$ are given by \eqref{eq:c} and \eqref{eq:s}, respectively. Integrating this on $\RR^L$ we get
\[
\int_{\RR^L} \osc(P_{T_0}f,B_{\omega}(\cdot)) \leq (1+ c_Ls^{\alpha}\omega^{\alpha})\int_{\RR^L}\osc(f,B_{s\omega}(\cdot))+ c_Ls^{\alpha}\omega^{\alpha}\int_{\RR^L}|f|,
\]
giving
\[
|P_{T_0}f|_{\alpha} \leq (1+c_Ls^{\alpha}\omega_0^{\alpha})s^{\alpha}|f|_{\alpha}+ c_Ls^{\alpha}\|f\|_1
\]
which gives \eqref{eq:LY2a} provided that $\omega_0$ is small enough. We note that choosing $\omega_0^{\alpha} \sim L^{-1/2}$ allows to choose $ c_Ls^{\alpha}\omega_0^{\alpha}$, thus $\sigma$ and $K$ uniformly in $L$. Assume from now on that we fix a choice of $\bar c$ and let $\omega_0$ such that $\omega_0^{\alpha}= \bar cL^{-1/2}$, $ c_Ls^{\alpha}\omega_0^{\alpha}=\delta$ for some fixed $\delta >0$ sufficiently small so that $(1+\delta)s^{\alpha}=:\sigma < 1$.

To prove \eqref{eq:LY2b}, we are going to use that for $f,g \in V_{\alpha}$
\begin{equation} \label{eq:Saus1}
\|fg\|_{\alpha} \leq |f|_{\alpha}\|g\|_{\infty}+\|f\|_{1}|g|_{\alpha}+\|f\|_{1}\|g\|_{\infty},
\end{equation}
as argued in the proof of Proposition 3.4, \cite{S00}. Let $f \in V_{\alpha}$ and $g=1_{X_{\varepsilon}^0}$. Recall that $X_{\eps}^0=X\backslash H_{\eps}$. Notice that $\osc(g,B_{\omega}(x)) = 1$ if $x \in B_{\omega}(\partial X_{\varepsilon}^0)=B_{\omega}(\partial H_{\varepsilon})$. Recall also that $H_{\eps}=\cup_{\v \in V^+}\cup_{\p \in \Lambda} H_{\varepsilon}^{\p,\v}$, where the union is not necessarily disjoint. Nevertheless, $\osc(g,B_{\omega}(x))$ is either 1 or 0 for $x \in B_{\omega}(\partial H_{\varepsilon}^{\p,\v})$ for each $\p \in \Lambda$ and $\v \in V^+$, and 0 elsewhere. This implies that
\begin{align*}
\int_{\RR^{ L}}\osc(g,B_{\omega}(x))dx & \leq\sum_{\v \in V^+}\sum_{\p \in \Lambda}  m_L(B_{\omega}(\partial H_{\varepsilon}^{\p,\v})) \\
&\leq  Ld((\varepsilon+2\omega)^2-{(\varepsilon-2\omega)^2})=  8Ld\varepsilon\omega,
\end{align*}
and thus
\[
\sup_{0 < \omega \leq \omega_0}\omega^{-\alpha}\int_{\RR^{ L}}\osc(g,B_{\omega}(x))dx \leq  8 Ld\varepsilon\omega_0^{1-\alpha}.
\]
So $g \in V_{\alpha}$ indeed, and by \eqref{eq:Saus1} we obtain
\begin{align*}
\|f1_{X_{\varepsilon}^0}\|_{\alpha}& \leq 
|f|_{\alpha}\|1_{X_{\varepsilon}^0}\|_{\infty}+\|f\|_1 |1_{X_{\varepsilon}^0}|_{\alpha}+\|f\|_1 \|1_{X_{\varepsilon}^0}\|_{\infty}
\\
&\leq |f|_{\alpha} + \|f\|_1 + \|f\|_1 |1_{X_{\varepsilon}^0}|_{\alpha} \\
&\leq \|f\|_{\alpha}+8 Ld\varepsilon\omega_0^{1-\alpha}\|f\|_{1}.
\end{align*}
\end{proof}

 Our choice of $\omega_0$ implied that $\omega_0^{\alpha}L^{1/2} \sim 1$. In turn, $L\omega_0^{1-\alpha} \sim L^{\frac{3}{2}-\frac{1}{2\alpha}}<1$ for sufficiently small $\alpha$. We now fix the value of $\alpha$ such that $\frac{3}{2}-\frac{1}{2\alpha} < 0$, that is, $\alpha \leq \min\{\beta,1/3\}$. Then $B(L,\eps,\omega_0) \leq 2\sigma+O(\eps)+K$, in particular
\begin{equation} \label{eq:B}
B(L,\eps,\omega_0) \leq C \quad \text{for $\eps$ sufficiently small.}
\end{equation}

\begin{remark}\label{Re:2}
Note that $f\ge 0$ with $\int_Xfdx=1$, there exists an $n\ge 0$, such that
$$\|\hat{P}^n_{\varepsilon}f\|_1 =\int_{X}P^n_{T_{0}}(1_{X^{n-1}_{\eps}} f)dx=\int_{X}1_{X^{n-1}_{\eps}} fdx<\int_{X} fdx=\|f\|_1.$$
\end{remark}
\begin{lemma}\label{lem:A4} Let $C(L,\omega_0)$ be the constant defined by \eqref{eq:CL}. Then
$$\sup_{\|f\|_{\alpha}\le 1} \|(\hat{P}_{\varepsilon}-P_{T_0})f\|_1\le C(L,\omega_0)\cdot Ld\eps^2.$$
\end{lemma}
\begin{proof}
Using  \eqref{eq:Sausinfty} we get:
\begin{equation*}
    \begin{split}
      \|(\hat{P}_{\varepsilon}-P_{T_0})f\|_1&= \| P_{T_0}(1_{X_\eps }f)-P_{T_0}f\|_1\le||1_{ H_\eps}f||_1\\
      &\le | H_\eps|\cdot\|f\|_\infty\le C(L,\omega_0) | H_\eps|\cdot\|f\|_{\alpha}.
    \end{split}
\end{equation*}
The proof is then completed by noting that $| H_\eps|= \left|\cup_{\v \in V^+}\cup_{\p \in \Lambda} H_{\varepsilon}^{\p,\v}\right|\leq Ld\eps^2$.
\end{proof}
 
\begin{proof}[Proof of Theorem \ref{thm:global}]
By Lemma \ref{lem:A2A3} and the fact that the unit ball of $(V_{\alpha},\|\cdot\|_{\alpha})$ is compact in the weak norm $\|\cdot\|_1$, the result of \cite{He} implies the residual spectrum of each operator $\hat{P}_{\varepsilon}$, $\eps\ge 0$ is contained in $\{z \in \mathbb{C}: |z| \leq \sigma\}$. The map $T_0$ admits an invariant density $\rho_0(x)=\Pi_{i=1}^{L} h(x_i)$.  Since $h$  is mixing, $\rho_0$ is mixing. This with the quasi-compactness of $P_0$ implies $P_0$ has a spectral gap when acting on $V_\alpha$. By Lemmas \ref{lem:A2A3} and \ref{lem:A4}, and the abstract perturbation result of \cite{KL99} the operator $\hat{P}_{\eps}:V_\alpha\to V_\alpha$, $\eps>0$ has a spectral gap with a leading simple eigenvalue $\lambda_{\eps}$ and by Remark \ref{Re:2} $\lambda_\eps\in (\sigma,1)$ holds. Moreover, there is $\rho_{\eps}\in V_\alpha$, a probability Borel measure $\nu_{\eps}$ and linear operators $Q_{\eps}:V_\alpha\to V_\alpha$ such that
\begin{equation}\label{eq:specdec}
\lambda_{\eps}^{-1}\hat{P}_{\eps}=\rho_{\eps}\otimes\nu_{\eps}+ Q_{\eps}
\end{equation}
with
\begin{itemize}
\item 
$\hat{P}_{\eps} \rho_{\eps}=\lambda_{\eps}\rho_{\eps},\, \nu_{\eps}\hat{P}_{\eps}=\lambda_{\eps}\nu_{\eps},\, Q_{\eps}\rho_{\eps}=0,\, \nu_{\eps}Q_{\eps}=0$;
\item $\sum_{n=0}^{\infty}\sup_{\eps}\| Q_{\eps}^n\|_{\alpha}<\infty$;
\item $\int\rho_{\eps}dx=1$ and $\sup_{\eps}\|\rho_{\eps}\|_{\alpha}<\infty$.
\end{itemize}
We assume the normalization $\nu_{0}(\rho_{\eps})=1$ for all $\eps \geq 0$.

We now state some more specific bounds for the above defined quantities. Observing the second item, the spectral radius of $Q_{\eps}$ is some $K_\eps(L)\in (0,1)$, so we have
\begin{equation} \label{eq:Qbound}
 \sum_{k=0}^{\infty}\|Q_{\eps}^k\|_\alpha\le \frac{1}{1-K_\eps(L)}\le  \frac{1}{1-K(L)}
\end{equation}
for some $K(L)$ that is independent of $\eps.$
Moreover, for the last item, since $\sigma$ is a uniform upper bound on the essential spectral radii of $\hat P_\eps$, we can fix $r\in(\sigma,1)$ and $\delta_0\in(0,1-r)$, both $r$ and $\delta_0$ are independent of $L$, so that $B(0,r)$ contains the essential spectrum of $\hat P_\eps$ and $\lambda_\eps\in B(1,\delta_0)$. Then using the fact that  $\rho_\eps$ is the non-negative eigenfunction corresponding to the dominant real eigenvalue $\lambda_\eps$, and $\int\rho_\eps=1$, by Lemma \ref{lem:A2A3}, we have
$$\lambda_\eps \|\rho_\eps\|_\alpha=\|\hat P_\eps\rho_\eps\|_\alpha\le\sigma\|\rho_\eps\|_\alpha+ B(L, \eps,\omega_0).$$
Therefore,
\begin{equation} \label{eq:rhobound}
\|\rho_\eps\|_\alpha \le\frac{ B(L, \eps,\omega_0)}{\lambda_\eps-\sigma}\le \frac{ C}{r-\sigma} \quad \text{for $\eps$ sufficiently small}
\end{equation}
by \eqref{eq:B}.

To obtain the first order approximation of $\lambda_\eps$ we proceed as in \cite{KL09'}. Notice that 
\[
   \int_X(\hat{P}_{\varepsilon}-P_{T_0})f dx=\int_X 1_{ H_\eps}f dx. 
\]
Consequently,
\begin{equation*}
\eta_{\eps}:=\sup_{\|f\|_{\alpha}\le 1}|\int_X(\hat{P}_{\varepsilon}-P_{T_0})f dx|\le  C(L,\omega_0)\int_X 1_{ H_\eps}dx \to 0\text{ as }\eps\to 0,
\end{equation*}
where we used that $\|f\|_{\infty} \leq C(L,\omega_0)\|f\|_{\alpha}$ by inequality \eqref{eq:Sausinfty}. In particular we obtained that
\begin{equation} \label{eq:etaO}
\eta_{\eps} =O(C(L,\omega_0)L\eps^2)
\end{equation}

As $X^0_{\eps}\cup H_{\eps} = X$, it holds that $|1_{H_{\eps}}|_{\alpha}=|1_{X^0_{\eps}}|_{\alpha}$. Thus we obtain the analogue of \eqref{eq:LY2b}
\[
\|1_{H_{\eps}}f\|_{\alpha}\leq \|f\|_{\alpha}+ 8Ld\eps\omega_0^{1-\alpha}\|f\|_1,
\]
and in particular
\[
\|1_{H_{\eps}}\rho_0\|_{\alpha}\leq \|\rho_0\|_{\alpha}+ 8Ld\eps\omega_0^{1-\alpha}.
\]
Using this and \eqref{eq:LY2a} we obtain
 \begin{equation*}
\begin{split}
    \|P_{T_0}(1_{ H_\eps}\rho_0)\|_{\alpha} &\leq \sigma\|1_{ H_\eps}\rho_0\|_\alpha+K\\
    &\le \sigma\left(\|\rho_0\|_\alpha +8Ld\eps\omega_0^{1-\alpha}\right) +  K\\
    &= \sigma \|\rho_0\|_\alpha + B(L,\varepsilon,\omega_0)\\
    &\le \frac{C\sigma}{r-\sigma} + C=\frac{Cr}{r-\sigma}=:C_1
\end{split}    
\end{equation*}
Recall further that $\underline{C}:=\inf_x h(x)>0$. Then
\begin{align*}
&\eta_{\eps}\cdot \|(\hat{P}_{\varepsilon}-P_{T_0})\rho_0\|_{\alpha}\\
&\le  C(L,\omega_0)\int_X 1_{H_\eps}dx\cdot \|P_{T_0}(1_{ H_\eps}\rho_0)\|_{\alpha} \le  C_1C(L,\omega_0) \int_X 1_{ H_\eps}dx \\
&\le\frac{ C_1C(L,\omega_0)}{\underline{C}}\int_X 1_{ H_\eps}\rho_0dx=\frac{ C_1C(L,\omega_0)}{\underline{C}}\cdot\int_X(\hat{P}_{\varepsilon}-P_{T_0})\rho_0 dx 
\end{align*}
Thus, we obtained
\begin{equation} \label{eq:eta2O}
\eta_{\eps}\cdot \|(\hat{P}_{\varepsilon}-P_{T_0})\rho_0\|_{\alpha} = O(C(L,\omega_0)) \int_X(\hat{P}_{\varepsilon}-P_{T_0})\rho_0 dx.
\end{equation}
We now give a bound on $|1-\nu_{\eps}(\rho_0)|$ following the proof of \cite[Lemma 6.1 (a)]{KL09'}. We will repeatedly use the fact that $(1-\lambda_{\eps}^{-1}\hat P_{\eps})(\lambda_{\eps}^{-1}\hat P_{\eps})^k(\rho_0)=(1-\lambda_{\eps}^{-1}\hat P_{\eps})Q_{\eps}^k(\rho_0)$ for all $k \geq 0$.

 Recall the normalization $\nu_0(\rho_{\varepsilon})=1$ for all $\varepsilon \geq 0$, and that $\lambda_{\eps}\hat{P}_{\eps}=\rho_{\eps} \otimes \nu_{\eps}+Q_{\eps}$. Then for all $n \in \NN$
\[
\nu_{\eps}(\rho_0)\rho_{\eps}=(\lambda_{\eps}\hat{P}_{\eps})^n (\rho_0)-Q^n_{\eps}(\rho_0), \\
\]
which implies 
\[
\nu_{\eps}(\rho_0)=\nu_0((\lambda_{\eps}\hat{P}_{\eps})^n (\rho_0)-Q^n_{\eps}(\rho_0)).
\]
Thus we can write
\begin{align*}
|1-\nu_{\eps}(\rho_0)|&=\lim_{n \to \infty}|\nu_0(\rho_0-(\lambda_{\eps}^{-1}\hat P_{\eps})^n(\rho_0))|  \\
&= \left |\sum_{k=0}^{n-1} \nu_0([(\lambda_{\eps}^{-1}\hat P_{\eps})^{k}-(\lambda_{\eps}^{-1}\hat P_{\eps})^{k+1}](\rho_0)) \right| \\
&\leq \sum_{k=0}^{n-1}| \nu_0([(1-\lambda_{\eps}^{-1}\hat P_{\eps})(\lambda_{\eps}^{-1}\hat P_{\eps})^k](\rho_0))| \\
&\leq \sum_{k=0}^{\infty}|\nu_0((1-\lambda_{\eps}^{-1}\hat P_{\eps})Q_{\eps}^k(\rho_0))|
\end{align*}
as $(1-\lambda_{\eps}^{-1}\hat P_{\eps})(\lambda_{\eps}^{-1}\hat P_{\eps})^k=(1-\lambda_{\eps}^{-1}\hat P_{\eps})Q_{\eps}^k$.
Thus
\begin{align*}
|1-\nu_{\eps}(\rho_0)| & \leq \sum_{k=0}^{\infty}|\nu_0((1-\lambda_{\eps}^{-1}\hat P_{\eps})Q_{\eps}^k(\rho_0))| \\
&\leq \sum_{k=0}^{\infty}|\nu_0((P_{T_0}-\hat P_{\eps})Q_{\eps}^k(\rho_0))|+\sum_{k=0}^{\infty}|\nu_0((\hat{P}_{\eps}-\lambda_{\eps}^{-1}\hat P_{\eps})Q_{\eps}^k(\rho_0))| \\
&\leq \sum_{k=0}^{\infty}|\nu_0((P_{T_0}-\hat P_{\eps})Q_{\eps}^k(\rho_0))|+|\lambda_{\eps}-1|\sum_{k=0}^{\infty}|\nu_0((\lambda_{\eps}^{-1}\hat P_{\eps})Q_{\eps}^k(\rho_0))| \\
&= \sum_{k=0}^{\infty}|\nu_0((P_{T_0}-\hat P_{\eps})Q_{\eps}^k(\rho_0))|+|\lambda_{\eps}-1|\sum_{k=0}^{\infty}|\nu_0(Q_{\eps}^{k+1}(\rho_0))| \\
&\leq \eta_{\eps} \sum_{k=0}^{\infty}\|Q_{\eps}^k\|_\alpha\cdot \|\rho_0\|_\alpha+ |1-\lambda_{\eps}|\|\nu_0\|_{\alpha} \sum_{k=1}^{\infty}\|Q_{\eps}^k\|_\alpha\cdot \|\rho_0\|_\alpha \\
&\leq \frac{C}{1-K(L)}\left(\eta_{\eps} + |1-\lambda_{\eps}|\right)
\end{align*}
Since $|1-\lambda_{\eps}|=|\nu_0((P_{T_0} -\hat P_{\eps})(\rho_{\eps}))|\leq \|\rho_{\eps}\|_{\alpha}\eta_{\eps} \leq C\eta_{\eps}$ (by \eqref{eq:rhobound}) we obtain that
\[
|1-\nu_{\eps}(\rho_0)| \leq \frac{C\eta_{\eps}}{1-K(L)}.
\]
Finally, by \eqref{eq:Qbound} and \eqref{eq:rhobound} we have
\begin{equation} \label{eq:1minmu0}
|1-\nu_{\eps}(\rho_0)|=O\left(\frac{C(L,\omega_0)L\eps^2}{1-K(L)}\right)
\end{equation}
by \eqref{eq:etaO}.

In what follows, we bound $\|Q_{\eps}^N(\rho_0)\|_{\alpha}$, $N \geq 0$ with the help of the tail-sum $\kappa_N=\sum_{n=N}^{\infty}\sup_{\eps}\|Q_{\eps}^n\|_{\alpha}$. The argument is analogous to \cite[Lemma 6.1 (b)]{KL09'}.
\begin{align*}
&\|Q_{\eps}^N(\rho_0)\|_{1} \\
&\leq \limsup_{n \to \infty}\|Q_{\eps}^N(\rho_0-(\lambda_{\eps}^{-1}\hat P_{\eps})^n(\rho_0))\|_{1}+\limsup_{n \to \infty}\|Q_{\eps}^{N+n}(\rho_0)\|_{\alpha} \\
&\leq \limsup_{n \to \infty}\left \|Q_{\eps}^N\left(\sum_{k=0}^{n-1}(\lambda_{\eps}^{-1}\hat P_{\eps})^k\rho_0-(\lambda_{\eps}^{-1}\hat P_{\eps})^{k+1}(\rho_0)\right)\right\|_{1}+\limsup_{n \to \infty}\kappa_{N+n}\|\rho_0\|_{\alpha} \\
&\leq \sum_{k=0}^{\infty} \|Q_{\eps}^N(\lambda_{\eps}^{-1}\hat P_{\eps})^k(1-\lambda_{\eps}^{-1}\hat P_{\eps})(\rho_0)\|_{1}+\limsup_{n \to \infty}\kappa_{N+n}\|\rho_0\|_{\alpha} \\
&\leq \sum_{k=0}^{\infty} \|Q_{\eps}^{N+k}(P_{T_0}-\hat P_{\eps}+(1-\lambda_{\eps}^{-1})\hat P_{\eps})(\rho_0)\|_{1} \\
&\leq \sum_{k=0}^{\infty} (\|Q_{\eps}^{N+k}(P_{T_0}-\hat P_{\eps})(\rho_0)\|_1+|1-\lambda_{\eps}^{-1}|\cdot \|Q_{\eps}^{N+k+1}(\rho_0)\|_{1}) \\
&\leq \sum_{k=0}^{\infty}\|Q_{\eps}^{N+k}\|_{\alpha}(\|(P_{T_0}-\hat P_{\eps})(\rho_0)\|_{\alpha}+|1-\lambda_{\eps}|\cdot \|\rho_0\|_{\alpha}),
\end{align*}
thus
\begin{equation} \label{eq:QNbound}
\|Q_{\eps}^N(\rho_0)\|_{1} = O(\kappa_N) (\|(P_{T_0}-\hat P_{\eps})(\rho_0)\|_{\alpha}+|1-\lambda_{\eps}|)
\end{equation}
by using \eqref{eq:rhobound}.

We now prove the formula for $\lambda_{\eps}$ as stated in item 1 of Theorem \ref{thm:global}. We follow the proof of \cite[Theorem 2.1]{KL09'} while tracking the $L$ dependence of the constants. First notice that $\nu_0((P_{T_0} - \hat P_{\eps})(\rho_0))=\mu_{0}(H_\eps)$.
\begin{align*}
&\nu_{\eps}(\rho_0)(1-\lambda_{\eps}) \\
&=(1-\nu_\eps(\rho_0))\nu_0((P_{T_0} - \hat P_{\eps})(\rho_\eps)) \\
&=\mu_{0}(H_\eps)-\nu_0((P_{T_0} - \hat P_{\eps})(1-(\lambda_{\eps}^{-1}\hat P_{\eps})^n)(\rho_0))-\nu_0((P_{T_0} - \hat P_{\eps})Q_{\eps}^n(\rho_0)) \\
&= \mu_{0}(H_\eps)-\sum_{k=0}^{n-1}\nu_0((P_{T_0} - \hat P_{\eps})(\lambda_{\eps}^{-1}\hat P_{\eps})^k(1-\lambda_{\eps}^{-1}\hat P_{\eps})(\rho_0))+O(\eta_{\eps}\|Q_{\eps}^n\rho_0\|_{1}) \\
&= \mu_{0}(H_\eps)-\sum_{k=0}^{n-1}\nu_0((P_{T_0} - \hat P_{\eps})(\lambda_{\eps}^{-1}\hat P_{\eps})^k(P_{T_0}-\hat P_{\eps})(\rho_0))\\
&+(1-\lambda_{\eps})\sum_{k=0}^{n}\nu_0((P_{T_0} - \hat P_{\eps})(\lambda_{\eps}^{-1}\hat P_{\eps})^k(\rho_0))\\
&+O(\kappa_n)(O(C(L,\omega_0))\mu_{0}(H_\eps)+\eta_{\eps}|1-\lambda_{\eps}|)
\end{align*}
where for the error term we used \eqref{eq:QNbound}, then \eqref{eq:eta2O}. Continuing,
\begin{align*}
\nu_{\eps}(\rho_0)(1-\lambda_{\eps})&=\mu_{0}(H_\eps)\left(1-\sum_{k=0}^{n-1}\lambda_{\eps}^{-k}q_{k,\eps} \right) \\
&+O(\eta_{\eps})|1-\lambda_{\eps}|\sum_{k=0}^{n-1}(|\nu_{\eps}(\rho_0)| \cdot \|\rho_{\eps}\|_1+\|Q_{\eps}^k\rho_0\|_1)\\
&+O(\kappa_n)(O(C(L,\omega_0))\mu_{0}(H_\eps)+\eta_{\eps}|1-\lambda_{\eps}|)
\end{align*}
where we used the notation
\[
q_{k,\eps}=\frac{\nu_0((P_{T_0} - \hat P_{\eps})\hat P_{\eps}^k(P_{T_0}-\hat P_{\eps})(\rho_0))}{\nu_0((P_{T_0} - \hat P_{\eps})(\rho_0))}.
\]
Using \eqref{eq:1minmu0}, \eqref{eq:rhobound} and \eqref{eq:Qbound}, we can write this as
\begin{align}
 \left(1+O\left(\frac{C\left(L,\omega_0\right)L\eps^2}{1-K(L)}\right)\right)&(1-\lambda_{\eps})(1+nO(\eta_{\eps}))\nonumber\\
&=\mu_{0}(H_\eps)\left(1-\sum_{k=0}^{n-1}\lambda_{\eps}^{-k}q_{k,\eps} \right) 
+O(\kappa_n)C(L,\omega_0)\mu_{0}(H_\eps) \nonumber\\
 \left(1+O\left(\frac{C\left(L,\omega_0\right)L\eps^2}{1-K(L)}\right)\right)(1&-\lambda_{\eps})(1+nO(C(L,\omega_0)L\eps^2))\nonumber \\
&=\mu_{0}(H_\eps)\left(1-\sum_{k=0}^{n-1}\lambda_{\eps}^{-k}q_{k,\eps} \right)
+O(\kappa_n)C(L,\omega_0)\mu_{0}(H_\eps) \label{eq:Lerror}
\end{align}
Assuming the limits $\lim_{\eps \to 0}q_{k,\eps}=:q_k$ exist, this implies that
\[
\lim_{\eps \to 0}\frac{1-\lambda_{\eps}}{\mu_{0}(H_\eps)}=1-\sum_{k=0}^{n-1}q_{k}+O(\kappa_n),
\]
and finally taking $n \to \infty$ gives the formula
\begin{equation} \label{eq:final}
\lim_{\eps \to 0}\frac{1-\lambda_{\eps}}{\mu_{0}(H_\eps)}=\theta
\end{equation}
by using the notation $\theta=1-\sum_{k=0}^{\infty}q_{k}$.

We now compute the values of $q_{k}$, $k \geq 0$ to obtain an explicit formula for $\theta$. To this end, we recall from \cite{K12} that $q_{k,\eps}$ can be computed as
\begin{align}
q_{k,\varepsilon}&=\frac{1}{\mu_0(H_{\varepsilon})}\int_{H_{\varepsilon}}\hat{P}_{\varepsilon}^kP_{T_0}(1_{H_{\varepsilon}}\rho_0)dx \nonumber \\
&=\frac{\mu_0(H_{\varepsilon}\cap T_0^{-1}H_{\varepsilon}^c\cap \dots T_0^{-k}H_{\varepsilon}^c \cap T_0^{-(k+1)}H_{\varepsilon})}{\mu_0(H_{\varepsilon}) } \label{eq:qkeps}
\end{align}

Recall further that $H_{\eps}=\cup_{\v \in V^+}\cup_{\p \in \Lambda}H_{\varepsilon}^{\p,\v}$ where 
$$
H_{\varepsilon}^{\p,\v}=\{x_{\p} \in A_{\varepsilon,\v}, x_{\p+\v} \in A_{\varepsilon,-\v}, x_{\q} \in I \text{ for } \q \neq \p, \p+\v \},
$$
and $I_{\p,\v}=\lim_{\varepsilon \to 0}H_{\varepsilon}^{\p,\v}$; i.e.,
\[
I_{\p,\v}=\{x_{\p}=a_\v, x_{\p+\v}= a_{-\v}, x_{\q} \in I \text{ for } \q \neq \p, \p+\v\},
\]
where $a_{\v}=\lim_{\varepsilon \to 0}A_{\varepsilon,\v}$ and $a_{-\v}=\lim_{\varepsilon \to 0}A_{\varepsilon,-\v}$.

Consider first the case when none of the elements of $S=\{(a_{\v}, a_{-\v})\}_{\v \in V^+}$ are periodic points, that is, $S^{per}=\emptyset$. Then by the expression \eqref{eq:qkeps} it is evident that in the $\eps \to 0$ limit we get $q_k=0$ for all $k \geq 0$. As a consequence, $\theta=1$. 

For each $(a_{\v},a_{-\v}) \in S^{per}$ recall that $k(\v)$ is the smallest integer such that $T_{2,0}^{k+1}(a_{\v},a_{-\v}) \in S$, and $V_k^+ \subset V^+$ is the set of directions for which $k(\v)=k$ and $K=\{k(\v)\}_{\v \in V^+}$. Notice that $S$ can be divided into periodic orbits, where the first preimage of each point intersecting $S$ is well-defined. We will denote the said preimage of $(a_{\v},a_{-\v})$ by $(\tau_*^{-(k'+1)}(a_{\v}),\tau_*^{-(k'+1)}(a_{-\v}))$. Consider now some $k \in K$. Computing $q_{k,\varepsilon}$, we get
\begin{align*}
q_{k,\varepsilon}&=\frac{1}{\mu_0(H_{\varepsilon})}\int_{H_{\varepsilon}}P_{\varepsilon}^kP_{T_0}(1_{H_{\varepsilon}}\rho_0)dx \\
&=\frac{1}{\mu_0(H_{\varepsilon})}\int_{H_{\eps}^k}\sum_{y \in T_0^{-(k+1)}(x)}\frac{\rho_0 (y)}{\|D T_0^{k+1}y\|}dx \\
&=\frac{1}{\mu_0(H_{\varepsilon})}\int_{H_{\eps}^k}\sum_{y \in T_0^{-(k+1)}(x)}\frac{\prod_{\p \in \Lambda}h(y_{\p})}{\prod_{\p \in \Lambda}|(\tau^{k+1})'(y_{\p})|}dx
\end{align*}
where $H_{\eps}^k=H_{\varepsilon}\cap T_0(H_{\varepsilon}^c)\cap\dots\cap T_0^{k}(H_{\varepsilon}^c)\cap T_0^{k+1}(H_{\varepsilon})$.

As a notational ease write $$F(x)=\sum_{y \in T_0^{-(k+1)}(x)}\frac{\prod_{\p \in \Lambda}h(y_{\p})}{\prod_{\p \in \Lambda}|(\tau^{k+1})'(y_{\p})|}.$$
We remark that $F=P_{T_0^{ k+1}}\rho_0=\rho_0$, but we prefer to keep this expression for reasons that will become clear in what follows. 

Recall that we defined $\Lambda^{\ell}=\{(\p_1,\dots,\p_k): \p_i \in \Lambda,\: i=1,\dots,\ell\}$,  $(V^+)^{\ell}=\{(\v_1,\dots,\v_{\ell}): \v_i \in V^+,\: i=1,\dots,\ell\}$ and let $(V_k^+)^{\ell}=\{(\v_1,\dots,\v_{\ell}): \v_i \in V_k^+,\: i=1,\dots,\ell\}$. Recalling \eqref{eq:pmF} we can write
\begin{align} \label{eq:pmF2}
\int_{H_{\eps}^k} F(x) dx&=\sum_{\ell=1}^{\lfloor {L}/2 \rfloor}(-1)^{\ell+1}\sum_{\substack{\{\p_i\}_{i=1}^{\ell} \in \Lambda^{\ell} \\ \p_i  < \p_{i+1}}}\sum_{\{\v_i\}_{i=1}^{\ell} \in (V^+)^{\ell}}\int_{(\cap_{i=1}^{\ell} H_{\eps}^{\p_i,\v_i})\cap H_{\eps}^k}F(x)dx.
\end{align}

Let us introduce the notation $x_{\neq \{\p_i,\p_i+\v_i\}_{i=1}^{\ell}}$ for the $(L-2\ell)$-dimensional vector that we get from $\{x_{\q}\}_{\q \in \Lambda}$ by discarding the coordinates $x_{\{\p_i,\p_i+\v_i\}_{i=1}^{\ell}}$. Assume $\{\v_i\}_{i=1}^{\ell} \in (V_k^+)^{\ell}$. Then $(\cap_{i=1}^{\ell} H_{\eps}^{\p_i,\v_i}) \cap H_{\eps}^k=\cap_{i=1}^{\ell} H_{\eps}^{\p_i,\v_i}$ for small enough $\varepsilon$. We now compute the limit
\[
\lim_{\eps \to 0}\frac{\mu_0(\cap_{i=1}^{\ell} H_{\eps}^{\p_i,\v_i})}{\mu_0(H_{\eps})} \cdot \frac{1}{\mu_0(\cap_{i=1}^{\ell} H_{\eps}^{\p_i,\v_i})}\int_{\cap_{i=1}^{\ell} H_{\eps}^{\p_i,\v_i}}F(x)dx.
\]
Since $F$ is continuous at $x \in S$ (based on the continuity assumptions and $\tau$ and $\tau'$, implying the continuity of $h$), we obtain
\begin{align*}
&\lim_{\eps \to 0} \frac{1}{\mu_0(\cap_{i=1}^{\ell} H_{\eps}^{\p_i,\v_i})}\int_{\cap_{i=1}^{\ell} H_{\eps}^{\p_i,\v_i}}F(x)dx \\
&=\prod_{i=1}^{\ell}\frac{h(\tau_*^{-(k+1)}(a_{\v_i}))h(\tau_*^{-(k+1)}(a_{-\v_i}))}{|(\tau^{k+1})'(\tau_*^{-(k+1)}(a_{\v_i}))(\tau^{k+1})'(\tau_*^{-(k+1)}(a_{-\v_i}))|}\cdot\frac{1}{h(a_{\v_i})h(a_{-\v_i})} \\
&\times \int\displaylimits_{I^{{L}-{2}\ell}}\sum_{y \in T_0^{-(k+1)}(x)}\prod_{\substack{\q \in \Lambda \\ \q \notin \{\p_i,\p_i+\v_i\}_{i=1}^{\ell}}}\frac{h(y_\q)}{|(\tau^{k+1})'(y_{\q})|}dx_{\neq \{\p_i,\p_i+\v_i\}_{i=1}^{\ell}} \\
&\times  \chi(\cap_{i=1}^{\ell} I_{\p_i,\v_i})
\end{align*}
where $\chi(A)=1$ if $A$ is nonempty and $\chi(A)=0$ otherwise. Notice that
\[
\int\displaylimits_{I^{{L}-{ 2}\ell}}\sum_{y \in T_0^{-(k+1)}(x)}\prod_{\substack{\q \in \Lambda \\ \q \notin \{\p_i,\p_i+\v_i\}_{i=1}^{\ell}}}\frac{h(y_\q)}{|(\tau^{k+1})'(y_{\q})|}dx_{\neq \{\p_i,\p_i+\v_i\}_{i=1}^{\ell}}=1
\]
by integrating successively according to each variable $x_{\q}$ and using that $P_{\tau}h=h$ and $\int_I h=1$. Then we can write
\begin{align*}
&\lim_{\eps \to 0} \frac{1}{\mu_0(\cap_{i=1}^{\ell} H_{\eps}^{\p_i,\v_i})}\int_{\cap_{i=1}^{\ell} H_{\eps}^{\p_i,\v_i}}F(x)dx \\
&=\prod_{i=1}^{\ell}\frac{h(\tau_*^{-(k+1)}(a_{\v_i}))h(\tau_*^{-(k+1)}(a_{-\v_i}))}{|(\tau^{k+1})'(\tau_*^{-(k+1)}(a_{\v_i}))(\tau^{k+1})'(\tau_*^{-(k+1)}(a_{-\v_i}))|}\cdot\frac{1}{h(a_{\v_i})h(a_{-\v_i})} \\
&\times \chi(\cap_{i=1}^{\ell} I_{\p_i,\v_i}).
\end{align*}
Notice that $(\tau_*^{-(k+1)}(a_{\v_i}),\tau_*^{-(k+1)}(a_{-\v_i}))=(a_{\v_j},a_{-\v_j})$ for some $\v_j \in V_k^+$.

On the other hand, using \eqref{eq:pm} we obtain
\begin{align*}
&\lim_{\eps \to 0}\frac{\mu_0(\cap_{i=1}^{\ell} H_{\eps}^{\p_i,\v_i})}{\mu_0(H_{\eps})}\\
&=\lim_{\eps \to 0}\frac{\mu_0(\cap_{i=1}^{\ell} H_{\eps}^{\p_i,\v_i})}{\sum_{k=1}^{\lfloor {L}/2 \rfloor}(-1)^{k+1}\sum_{\substack{\{\p_i\}_{i=1}^{k} \in \Lambda^k \\ {\p_i < \p_{i+1}}}}\sum_{\{\v_i\}_{i=1}^{k} \in (V^+)^k}\mu_0\left(\bigcap_{i=1}^k H_{\eps}^{\p_i,\v_i}\right)} \\
&=\begin{cases}
\frac{h(a_{\v_1})h(a_{-\v_{1}})}{\sum_{\p \in \Lambda}\sum_{\v \in V^+}h(a_{\v})h(a_{-\v})} & \quad \text{if $\ell =1$,} \\
0 & \quad \text{otherwise.}
\end{cases}
\end{align*}
The reason we get zero in the limit for $\ell > 1$ is that  $m_{L-2}(\cap_{i=1}^{\ell} I_{\eps}^{\p_i,\v_i})=0$ while $m_{L-2}( I_{\eps}^{\p_i,\v_i}) > 0$ (we have $m_{L-2}( I_{\eps}^{\p_i,\v_i})=1$ in fact.) This implies \begin{align*}
&\lim_{\eps \to 0}\frac{\mu_0(\cap_{i=1}^{\ell} H_{\eps}^{\p_i,\v_i})/\eps^2}{\sum\limits_{k=1}^{\lfloor {L}/2 \rfloor}(-1)^{k+1}\sum\limits_{\substack{\{\p_i\}_{i=1}^{k} \in \Lambda^k \\ { \p_i < \p_{i+1}}}}\sum\limits_{\{\v_i\}_{i=1}^{k} \in (V^+)^k}\mu_0\left(\bigcap\limits_{i=1}^k H_{\eps}^{\p_i,\v_i}\right)/\eps^2} \\
&=\lim_{\eps \to 0}\frac{\mu_0(\cap_{i=1}^{\ell} H_{\eps}^{\p_i,\v_i})/\eps^{2^{\ell}}\cdot \eps^{2(\ell-1)}}{\sum\limits_{k=1}^{\lfloor {L}/2 \rfloor}(-1)^{k+1}\sum\limits_{\substack{\{\p_i\}_{i=1}^{k} \in \Lambda^k \\ {\p_i < \p_{i+1}}}}\sum\limits_{\{\v_i\}_{i=1}^{k} \in (V^+)^k}\mu_0\left(\bigcap\limits_{i=1}^k H_{\eps}^{\p_i,\v_i}\right)/\eps^{2^{k}}\cdot \eps^{2(k-1)}} \\
&=\lim_{\eps \to 0}\frac{\prod\limits_{i=1}^{\ell}h(a_{\v_i})h(a_{-\v_i}) \eps^{2(\ell-1)}}{\sum\limits_{k=1}^{\lfloor {L}/2 \rfloor}(-1)^{k+1}\sum\limits_{\substack{\{\p_i\}_{i=1}^{k} \in \Lambda^k \\ {\p_i < \p_{i+1}}}}\sum\limits_{\{\v_i\}_{i=1}^{k} \in (V^+)^k}\prod\limits_{i=1}^{k}h(a_{\v_i})h(a_{-\v_i}) \eps^{2(k-1)}} \\
&=\frac{0}{\sum_{\p \in \Lambda}\sum_{\v \in V^+}h(a_{\v})h(a_{-\v})+0}=0.
\end{align*}
To finish the proof, consider \eqref{eq:pmF2}, keep only the term $\ell =1$ and repeating this computation for each $\p \in \Lambda$ and $\v \in V_k^+$, we obtain
\begin{align*}
q_k&=\frac{1}{\sum_{\p \in \Lambda}\sum_{\v \in V^+}h(a_{\v})h(a_{-\v})}\sum_{\p \in \Lambda} \sum_{\v \in V_k^+}\frac{h(a_{\v})h(a_{-\v})}{|(\tau^{k+1})'(a_{\v})(\tau^{k+1})'(a_{-\v})|} \\
&=\frac{1}{\sum_{\v \in V^+}h(a_{\v})h(a_{-\v})} \sum_{\v \in V_{k}^+}\frac{h(a_{\v})h(a_{-\v})}{|(\tau^{k+1})'(a_{\v})(\tau^{k+1})'(a_{-\v})|}.
\end{align*}
Finally,
\begin{align}
\theta&=1-\sum_{k \in K} q_k \nonumber \\
&=1-\frac{1}{\sum_{\v \in V^+}h(a_{\v})h(a_{-\v})} \sum_{k \in K}\sum_{\v \in V_{k}^+}\frac{h(a_{\v})h(a_{-\v})}{|(\tau^{k+1})'(a_{\v})(\tau^{k+1})'(a_{-\v})|} \label{eq:thetap}.
\end{align}

Lemma \ref{lem:mu0H} readily implies that $\mu_0(H_\eps)={L}\Xi_{\eps}(1+o(1)) \sim \eps^2 {L} d$ as
\[
d\underline{C}^2\eps^2 \leq \sum_{\v \in V^+}\int_{A_{\eps,\v}}\int_{A_{\eps,-\v}}h(x_1)h(x_2)dx_1dx_2 \leq d\overline{C}^2\eps^2
\]
where $\overline{C},\underline{C}$ are the respective upper and lower bounds on $h$ noted previously.  The $o(1)$ error term is understood in the limit $L\eps^2 \to 0$, which for fixed $L$ is equivalent to $\eps \to 0$.

To prove the second item, first notice that

\begin{align*}
m_{L}(X_{\eps}^{n-1})&=\int_X 1_{X_{\eps}^{n-1}}(x)dx=\int_X P_{T_0}^{n}(1_{X_{\eps}^{n-1}}(x))dx=\int_X\hat{P}_{\eps}^{n}1(x)dx\\
&=\lambda^n_\eps\left[\int_X\rho_\eps\otimes\nu_{\eps}(1) dx + \int_X Q^n_\eps 1(x) dx\right],
\end{align*}
where in the last step we have used the spectral decomposition \eqref{eq:specdec}. Consequently,
$$
\frac{1}{n}\ln \left(m_{L}(X_{\eps}^{n-1})\right)-\ln\lambda_\eps=\frac1n\ln\left[\int_X\rho_\eps\otimes\nu_{\eps}(1) dx + \int_X Q^n_\eps 1(x) dx\right].
$$
Since $\|Q_\eps\|_1\le \|Q_\eps\|_{\alpha}<1$, $\exists C>0$ such that
$$
\left|\frac{1}{n}\ln \left(m_{L}(X_{\eps}^{n-1})\right)-\ln\lambda_\eps\right|\le\frac{C}{n}.$$
 Hence, 
 $$\lim_{n\to\infty}\frac{1}{n}\ln \left(m_{L}(X_{\eps}^{n-1})\right)=\ln\lambda_\eps$$
 and the first collision rate, with respect to $m_L$ in the $L$-dimensional system is given by $\hat r_L=-\ln\lambda_\eps$.
 Using the asymptotics of $\lambda_\eps$, we have 
\begin{equation*}
\begin{split}
\lambda_{\varepsilon}&=1-\mu_0(H_\eps)\cdot\theta(1+o(1))\\
&=e^{-\mu_0(H_\eps)\cdot\theta(1+ o(1))}
\end{split}
\end{equation*}
This implies
\begin{equation}\label{eq:Lrate}
\hat r_L= \mu_0(H_\eps) \cdot\theta(1+o(1)).
\end{equation}

%\begin{align*}
%\mstrike{\mu_{\eps}(X_{\eps}^n)}&=\mstrike{\int_X 1_{X_{\eps}^n}\rho_{\eps}(x)dx=\int_X P_{T_0}^{n+1}(1_{X_{\eps}^n}\rho_{\eps})(x)dx=\int_X\hat{P}_{\eps}^{n+1}\rho_{\eps}(x)dx}\\
%&=\mstrike{\lambda_{\eps}^{n+1},}
%\end{align*}
%\strike{which implies that the first collision rate in the $N$-dimensional system is given by}
%$$\mstrike{-\lim_{n \to\infty}\frac{1}{n}\ln \mu_{\eps}(X_{\eps}^n)=-\ln\lambda_\eps.}$$

The proof of the third item of the theorem follows from item 1 along the lines of Proposition 2 of \cite{K12}.

 We choose $N$ such that $N\eta_{\eps}+O(C(L,\omega_0)\kappa_N)$ is minimal. Denote
 \begin{equation} \label{eq:zeta}
 \zeta(L)=|\log (1-K(L))|
 \end{equation}
 where $K(L)$ was defined by \eqref{eq:Qbound}. Thus $N=O\left(\frac{1}{\zeta(L)} \cdot \left| \log \left( \frac{\eta_{\eps}}{C(L,\omega_0)\zeta(L)}\right)\right|\right)$. Set  $\delta_{\eps}=O\left(\frac{\eta_{\eps}}{\zeta(L)}\left| \log \frac{\eta_{\eps}}{C(L,\omega_0)\zeta(L)}\right|\right)$. 

Let $\xi_{\eps}=\theta_{N,\eps}+O(C(L,\omega_0)\kappa_N)$ where $\theta_{N,\eps}=\sum_{k=0}^{N-1} \lambda_{\eps}^{-k}q_{k,\eps}$. By \eqref{eq:Lerror} we obtain that
\begin{align*}
\lambda_{\eps}&=1-\mu_0(H_{\eps})\xi_{\eps} (1+O(N\eta_{\eps})) \\
&=exp[-\mu_0(H_{\eps})\xi_{\eps} (1+O(\delta_{\eps}))]
\end{align*}
As $\lambda_{\eps}=1-\mu_0(H_{\eps})(\theta+o(1))$ by \eqref{eq:final} we also have $\lim_{\eps \to 0}\xi_{\eps}=\theta > 0$, in particular $\xi_{\eps} > 0$ for $\eps$ small enough.

Let $n=\lfloor t/(\xi_{\eps}\mu_0(H_{\eps})) \rfloor$. First assume that $n > \left|\frac{\log \eta_{\eps}}{\zeta(L)} \right|$. Then as argued in \cite[Equation (26)]{K12}, we can write
\begin{align*}
\mu_0\Big\{t_\eps\ge\frac{t}{\xi_\eps\mu_0(H_\eps)}\Big\}&=\lambda_{\eps}^{n+q}(\nu_{\eps}(\rho_0)+O(\|Q_\eps^n\|_{\alpha})) \\
&=exp[-t(1+O(\delta_{\eps}))]\cdot \left(1+O\left(\frac{C\left(L,\omega_0\right)L\eps^2}{1-K(L)}\right)\right)
\end{align*}
where in the last line we used \eqref{eq:1minmu0}. So for $\eps$ small enough we have
\begin{align*}
\left|\mu_0\Big\{t_\eps\ge\frac{t}{\xi_\eps\mu_0(H_\eps)}\Big\}-e^{-t}\right| &\leq  C(t\vee 1)e^{-t} \delta_{\eps} \\
&\leq  C(t\vee 1)e^{-t}  \frac{\eta_{\eps}}{\zeta(L)}\log \frac{\eta_{\eps}}{C(L,\omega_0)\zeta(L)} \\
&\leq  C(t\vee 1)e^{-t} C(L,\omega_0) \frac{L\eps^2}{\zeta(L)}\log \frac{L\eps^2}{\zeta(L)}.
\end{align*}
The case of  $n \leq \left|\frac{\log \eta_{\eps}}{\zeta(L)} \right|$ can be argued in the same way as in the proof of \cite[Proposition 2]{K12}.
\end{proof}

\begin{proof}[Proof of Theorem \ref{thm:perunit}]
To derive the expression in the theorem, we consider 
$$
\frac{\hat r_L}{\mu_0(H_\eps)}\cdot\frac{\mu_0(H_\eps)}
{{L}\Xi_\eps}
$$
as $\eps\to 0$. By equation \eqref{eq:escape0} 
\begin{equation*}
\frac{\hat r_L}{\mu_0(H_\eps)}= \theta(1+o(1))
\end{equation*}
in the limit as $\eps\to 0$. Moreover, equation \eqref{eq:escape} implies  $$\frac{\mu_0(H_{\eps})}{{L}\Xi_{\eps}}=1+o(1)$$ as $\eps\to 0$. This completes the proof.
\end{proof}
We finish the paper by a remark on considering the infinite dimensional limit of the collision rate and its per lattice unit counterpart.
\begin{remark}\label{rem:infinite}
If we want to study collision rates in the infinite limit of the system; i.e., as ${L}\to \infty$, we need an asymptotic formula for $\lambda_{\eps}$ in some kind of joint limit where $L \to \infty$ and $\eps \to 0$. To this end, recall \eqref{eq:Lerror} :
\begin{align*}
\left(1+O\left(\frac{C\left(L,\omega_0\right)L\eps^2}{1-K(L)}\right)\right)&\frac{1-\lambda_{\eps}}{\mu_{0}(H_\eps)}(1+nO(C(L,\omega_0)L\eps^2))\nonumber \\
&=1-\sum_{k=0}^{n-1}\lambda_{\eps}^{-k}q_{k,\eps}
+O(\kappa_n)C(L,\omega_0)
\end{align*}
The term $O(\kappa_n)C(L,\omega_0)$ causes serious difficulties in taking any kind of limit in which $L \to \infty$, as $C(L,\omega_0)$, defined by \eqref{eq:CL} diverges as $L \to \infty$. We note that the constant $C(L,\omega_0)$ appears in that term due to \eqref{eq:eta2O}, the analogue of (A6) in \cite{KL09'}.  In addition, the term $K(L)$ (a bound on the modulus of the second largest eigenvalue of $\hat{P}_{\eps}$) is also difficult to control.

%We might try proceeding by choosing $n=C(L,\omega_0)$. As $\kappa_n=O(\|Q_{\eps}^n\|_{\alpha})=O(\sigma^n)$, this gives
%\begin{align*}
%(1+O(C(L,&\omega_0)L\eps^2))(1-\lambda_{L,\eps})(1+O(C(L,\omega_0)^2L\eps^2))\\
%&=\mu_{0}(H_\eps)\left(1-\sum_{k=0}^{n-1}\lambda_{L,\eps}^{-k}q_{k,\eps} \right) +O(\sigma^{C(L,\omega_0)}C(L,\omega_0))\mu_{0}(H_\eps). 
%\end{align*}
%Let $\overset{*}{\lim}$ be the limit where we take $\eps \to 0$ and $L \to  \infty$ such that
%\begin{equation} \label{eq:f}
%f(L,\eps)=C(L,\omega_0)^2L\eps^2=C(L,\bar c L^{-1/2})^2L\eps^2 \to 0.
%\end{equation}
%Assuming the $\lim^*\sum_{k=0}^{C(L,\bar c L^{-1/2})-1}\lambda_{L,\eps}^{-k}q_{k,L,\eps}=:1-\theta^*$ exists we obtain that $$\lim^* \frac{1-\lambda_{L,\eps}}{\mu_{0}(H_\eps)}=\theta^*.$$ We saw that $|1-\lambda_{L,\eps}| =O(\eta_{L,\eps})=O(C(L,\omega_0)L\eps^2)$, which provides us that $\lambda_{\eps}^{-(C(L,\bar c L^{-1/2})-1)}=1+O(C(L,\omega_0)^2L\eps^2) \to 0$ as $f(L,\eps) \to 0$. So we need to understand if $\lim^*\sum_{k=0}^{C(L,\bar c L^{-1/2})-1}q_{k,L,\eps}$ exists. Assume $\lim^* q_{k,L,\eps}=q_k^*$ exists. If we could prove that for each $\delta > 0$, $\max_{k \leq C(L,\bar c L^{-1/2})-1}|q_{k,L,\eps}-q_k^*| < \frac{\delta}{C(L,\bar c L^{-1/2})}$ would be guaranteed by $f(L,\eps)$ being sufficiently small, we would obtain that $$\theta^*=1-\sum_{k=0}^{\infty}q_k^*.$$ However, proving the existence of $q_k^*$ and the rate of convergence $q_{k,L,\eps} \to q_k^*$ does not seem tractable in the generality of the present work.
\end{remark}

    \end{document}